\documentclass[11pt,           % in 11pt fonts
english                        % default language
]{article}

\usepackage[english,strings]{babel}     % with explicit language
\usepackage[utf8]{inputenc}    % smart input of funny chars
\usepackage[T1]{fontenc}       % also for the font encoding
\usepackage{longtable}         % tables longer than one page
\usepackage{exscale}           % large summation signs in 11pt
\usepackage[final]{graphicx}   % to include pdf pictures
\usepackage[sort]{cite}        % nicer citations
\usepackage{array}             % nice tables
\usepackage{fancyhdr}          % nicer headers
\usepackage[a4paper]{geometry} % geometry of page layout
\usepackage{xspace}            % better spacing after macros
\usepackage{tikz}              % for commutative diagrams and stuff
\usepackage{tikzsymbols}       % for tikz symbols
\usepackage{nchairx}           % additional packages
\usepackage{wasysym}           % smiley symbols
\usepackage[sectionbib         % treat refs as section, not as chapter
           ]{chapterbib}       % separate bib for each chapter
\usepackage{appendix}
\usepackage[expansion=false    % no font expansion
           ]{microtype}        % only protrusion
\usepackage[nottoc]{tocbibind} % refs and index in the toc
\usepackage[backref=page,      % backrefs in the bibliography
           final=true,         % always treat as final
           pdfpagelabels       % use pdf page labels
           ]{hyperref}         % hyperrefs are cool!

\usepackage{tikz-cd}       % nice commutative diagrams
\usepackage{xcolor} % more colors
\usetikzlibrary{arrows,calc,through,backgrounds,matrix,
	decorations.pathmorphing,positioning,babel}    % useful libraries

\newcommand{\poly}{{\scriptscriptstyle{\mathrm{poly}}}}    % Index poly
\newcommand{\Tpoly}{T_\poly}  % multivector fields
\newcommand{\Tay}{{\scriptscriptstyle{\mathrm{Tay}}}}    % Index Tay for taylor
\newcommand{\nice}{{\scriptscriptstyle{\mathrm{nice}}}}      % Index nice  
\newcommand{\Cart}{{\scriptscriptstyle{\mathrm{Cart}}}} % index for Cartan

\title{The Strong Homotopy Structure of Poisson Reduction}
\date{April 2020}

\author{
  \textbf{Chiara Esposito}\thanks{\texttt{chesposito@unisa.it}}, \textbf{Andreas Kraft}\thanks{\texttt{akraft@unisa.it}},\\[0.3cm]
   Dipartimento di Matematica\\
   Università degli Studi di Salerno\\
   via Giovanni Paolo II, 123\\
   84084 Fisciano (SA)\\
   Italy \\[0.5cm]
  \textbf{Jonas Schnitzer}\thanks{\texttt{jonas.schnitzer@math.uni-freiburg.de}},\\[0.3cm]
  Department of Mathematics\\
   University of Freiburg\\
   Ernst-Zermelo-Straße, 1\\
	D-79104 Freiburg\\
   Germany
}

\begin{document}
%\selectlanguage{english}

% title page
\maketitle

% abstract
\abstract{ In this paper we propose a reduction scheme for multivector
  fields phrased in terms of $L_\infty$-morphisms. Using well-know
  geometric properties of the reduced manifolds we perform a Taylor
  expansion of multivector fields, which allows us to built up a
  suitable deformation retract of DGLA's. We first obtained an
  explicit formula for the $L_\infty$-Projection and -Inclusion of
  generic DGLA retracts.  We then applied this formula to the
  deformation retract that we constructed in the case of multivector
  fields on reduced manifolds.  This allows us to obtain the desired
  reduction $L_\infty$-morphism.  Finally, we perfom a comparison with
  other reduction procedures. }

\newpage

% table of contents
\tableofcontents

%
% section: Introduction
%

\section{Introduction}

This paper aims to propose a reduction scheme for multivector 
fields that is phrased in terms of $L_\infty$-morphisms 
and adapted to deformation quantization. 
Deformation quantization has been introduced in \cite{bayen.et.al:1978a} by
Bayen, Flato, Fronsdal, Lichnerowicz and Sternheimer and it relies on the
idea that the quantization of a Poisson manifold $M$
is described by a formal deformation of the commutative algebra of smooth
complex-valued functions $\Cinfty(M)$, a so-called \emph{star product}. 
The existence and classification of 
star products on Poisson manifolds has been provided by Kontsevich's
formality theorem \cite{kontsevich:2003a}, whereas the invariant 
setting of Lie group actions has been treated by 
Dolgushev, see~\cite{dolgushev:2005a,dolgushev:2005b}. In the last 
years many developments have been done, see e.g. \cite{calaque:2005a, calaque:2007a, liao:2019a}.
More explicitly, the formality provides an $L_\infty$-quasi-isomorphism between 
the differential graded Lie algebra of multivector fields and 
the multidifferential operators resp. the invariant versions. 
One open question and our main motivation is 
to investigate the compatibility of deformation quantization and 
phase space reduction in the Poisson setting.

In the classical setting one considers here the Marsden-Weinstein reduction 
\cite{marsden.weinstein:1974a}. Suppose that a Lie group 
$\group{G}$ acts by Poisson diffeomorphisms on the Poisson manifold $M$ 
and that it allows an $\Ad^*$-equivariant momentum map 
$J\colon M\longrightarrow \liealg{g}^*$ with $0\in \liealg{g}^*$
as value and regular value, where $\liealg{g}$ is the Lie algebra of
$\group{G}$. Then $C = J^{-1}(\{0\})$ is a closed embedded submanifold of
$M$ and the reduced manifold $M_\red = C/\group{G}$ is again a Poisson 
manifold if the action on $C$ is proper and free. Reduction theory is very important
and it is still very active field of research. Among the others, we mention the categorical
reformulation performed in \cite{dippell:2019a}.

In the setting of deformation quantization a quantum reduction scheme
has been introduced in \cite{bordemann.herbig.waldmann:2000a},
see also \cite{gutt.waldmann:2010a} for a slightly different formulation,
which allows the study of the compatibility between the reduction scheme 
and the properties of the star product, as in \cite{esposito:2019a}.
One crucial ingredient are quantum momentum maps (see~\cite{xu:1998a}) 
and pairs consisting of star products with compatible quantum 
momentum maps are called \emph{equivariant star products}.
For symplectic manifolds these equivariant star products have recently been 
classified and it has been shown that quantization commutes 
with reduction, see \cite{reichert.waldmann:2016a,reichert:2017a,reichert:2017b}. 
More precisely, equivariant star products on $M$ are classified by certain 
elements in the cohomology of the Cartan model for equivariant de Rham cohomology 
\cite{guillemin.sternberg:1999a} 
and the characteristic classes of the equivariant star product and 
the reduced star product are related by pull-backs. 

In the more general setting of Poisson manifolds, star products are 
classified by Maurer-Cartan elements in the DGLA of multivector fields, 
i.e. by formal Poisson 
structures. Unfortunately in this case there is no pull-back available and one has to 
use different techniques. Motivated by the aim of 
reducing the formality, we want to describe the reductions 
in terms of $L_\infty$-morphisms. In particular, in this paper we 
construct such a reduction for the classical side, i.e. for the 
\emph{equivariant multivector fields} $T_\liealg{g}(M)$, a certain DGLA whose 
Maurer-Cartan elements are invariant Poisson structures 
with equivariant momentum maps. Assuming for simplicity $M = C \times 
\liealg{g}^*$, which always holds locally in suitable situations, 
we can perform a Taylor 
expansion around $C$, obtaining a new DGLA $T_\Tay(C\times\liealg{g}^*)$. 
On $C\times \liealg{g}^*$ we have the canonical momentum map $J$ given 
by the projection on $\liealg{g}^*$ and the canonical linear Poisson structure 
$\pi_\KKS$ induced by the action Lie algebroid. They give 
a new DGLA structure on $T_\Tay(C\times \liealg{g}^*)$ with differential 
$[\pi_\KKS -J, \argument]$ and we show that this DGLA 
is quasi-isomorphic to the multivector fields on $M_\red$, as desired. 
One has an $L_\infty$-quasi-isomorphism between these two 
DGLA's, see Theorem~\ref{thm:ClassicalTredTaylor}:

\begin{nntheorem}
  There exists an $L_\infty$-quasi-isomorphism $\tilde{\mathrm{T}}_\red \colon 
	T_\Tay(C\times \liealg{g}^*) \rightarrow T_\poly(M_\red)$.
\end{nntheorem}
The morphism $\tilde{\mathrm{T}}_\red$ is obtained by inverting 
a certain inclusion $i$ of DGLA's. In order to give a more explicit formula  
we look at general deformation retracts: let $(A, \D_A)$ and 
$(B,\D_B)$ be two differential graded Lie algebras and assume 
that we have
\begin{equation}
  \label{eq:IntroDefRetract}
  \begin{tikzcd} 
	(A,\D_A)
	\arrow[rr, shift left, "i"] 
  &&   
  \left(B, \D_B\right)
  \arrow[ll, shift left, "p"]
	\arrow[loop right, "h"] 
\end{tikzcd} 
\end{equation}
where $i$ and $p$ are quasi-isomorphisms of cochain complexes with
homotopy $h$, and where $ p \circ i = \id_A$ and $h^2 = h \circ i = p \circ h = 0$. 
Using for a coalgebra morphism $F \colon \Sym(B[1]) \rightarrow \Sym(A[1])$ 
the notation
\begin{equation*}
  L_{\infty,k+1}(F)
  =
 	Q^1_{A,2} \circ F^2_{k+1} - F^1_k \circ Q^k_{B,k+1},
\end{equation*}
where $Q^k_{A,k+1}$ and $Q^k_{B,k+1}$ are the extensions 
of the Lie brackets to the symmetric algebras, and extending $h$ 
in an appropriate way to $H_k$  on $\Sym^k(B[1])$, we prove 
in Proposition~\ref{prop:Infinityprojection} and 
\ref{prop:Infinityinclusion}:

\begin{nnproposition}
  Given a deformation retract as in \eqref{eq:IntroDefRetract}.
  \begin{propositionlist}
  \item If $i$ is a DGLA morphism, then $P \colon \Sym^\bullet(B[1])
    \rightarrow \Sym^\bullet (A[1])$ with structure maps $P_1^1 = p$
    and $P_{k+1}^1 = L_{\infty,k+1}(P) \circ H_{k+1}$ for $k \geq 1$
    yields an $L_\infty$-quasi-isomorphism that is quasi-inverse to
    $i$.
  \item If $p$ is a DGLA morphism, then $I\colon \Sym^\bullet
    (A[1])\to \Sym^\bullet (B[1])$ with structure maps $I_1^1=i$ and
    $I_k^1=h\circ L_{\infty,k}(I)$ for $k\geq 2$ is an
    $L_\infty$-quasi-isomorphism that is quasi-inverse to $p$.
  \end{propositionlist}
\end{nnproposition}
This allows us to give a more explicit description of
$\tilde{\mathrm{T}}_\red$ and its $L_\infty$-quasi-inverse. Moreover,
it allows us to globalize the result, compare
Theorem~\ref{thm:ClassicalTred2}:

\begin{nntheorem}
  There exists a curved $L_\infty$-morphism 
	\begin{align}
	  \mathrm{T}_\red\colon 
		(T_{\liealg{g}}(M),\lambda,-[J,\argument],  [\argument,\argument])
	  \longrightarrow 
	  (\Tpoly(M_\mathrm{red}),0,0,[\argument,\argument]),
	\end{align}	   
	where the curvature $\lambda = e^i \otimes (e_i)_M$ is given by  
	the fundamental vector fields of the $\group{G}$-action. 
\end{nntheorem}
We call $\mathrm{T}_\red$ \emph{reduction $L_\infty$-morphism} and we
extend the statements to the setting of formal power series in
$\hbar$.  After rescaling the involved curvatures and differentials
appropriately, $\mathrm{T}_\red$ gives in particular a way to
associate formal Maurer-Cartan elements.  In
$T_\liealg{g}(M)[[\hbar]]$
resp. $T_\Tay(C\times\liealg{g}^*)[[\hbar]]$ with rescaled structures,
formal Maurer-Cartan elements can be interpreted as formal Poisson
structures $\pi_\hbar$ with formal momentum map $J_\hbar=J + \hbar
J'$.  Thus, we have the following properties:
\begin{itemize}
\item The Poisson bracket $\{\argument,\argument\}_\hbar$ induced by
  $\pi_\hbar$ is $\group{G}$-invariant,
			
\item The fundamental vector fields are given by $\xi_M = \{
  \,\cdot\,,J_\hbar(\xi)\}_\hbar \in \Secinfty(TM)$, and
				
\item $\{ J_\hbar(\xi), J_\hbar(\eta)\}_\hbar = J_\hbar([\xi,\eta])$.
\end{itemize}
Comparing the orders of $\hbar$ directly shows that the lowest order
is a well-defined Poisson structure on $M$ and that $J$ is an
equivariant momentum map with respect to it, and $\mathrm{T}_\red$
maps such an object to a formal Poisson structure on $M_\red$.

Note that there is also another reduction scheme for such formal
Poisson structures with formal momentum maps, obtained by adapting the
reduction scheme for star products from
\cite{bordemann.herbig.waldmann:2000a,gutt.waldmann:2010a}, i.e.
using the homological perturbation lemma \cite{crainic:2004a:pre}.
Finally, we show in Theorem~\ref{thm:FormalReductionsEqual}:

\begin{nntheorem}
  The reduction of formal equivariant Poisson structures with formal
  momentum maps via the reduction $L_\infty$-morphism coincides with
  the reduction of formal Poisson structures via the homological
  perturbation lemma.
\end{nntheorem}
%
%to recheck at the end:
The paper is organized as follows: In Section~\ref{sec:Preliminaries}
we recall the basic notions of (curved) $L_\infty$-algebras,
$L_\infty$-morphisms and twists. In
Section~\ref{sec:ExplicitFormulasDefRetracts} we consider general
deformation retracts of DGLA's and prove the explicit formulas for the
extensions of the inclusion resp. projection to $L_\infty$-morphisms
needed to describe $\tilde{\mathrm{T}}_\red$ in
Section~\ref{sec:TaylorSeriesofTpoly}. Here we also construct the
reduction scheme for the Taylor expansion, both in the classical and
the formal setting.  Finally, in Section~\ref{sec:ReductionMorphism}
we construct the global reduction $L_\infty$-morphism and compare the
reduction via $\mathrm{T}_\red$ with the classical Marsden-Weinstein
reduction and with the reduction of formal Poisson structures via the
homological perturbation lemma as explained in
Appendix~\ref{sec:BRSTlikeReduction}.

\subsection*{Acknowledgements}
  The authors are grateful to Alejandro Cabrera and Stefan Waldmann 
	for the helpful comments.
  This work was supported by the National Group for Algebraic and 
	Geometric Structures, and their Applications (GNSAGA – INdAM).
  The third author is supported by the DFG research training group 
  "gk1821: Cohomological Methods in Geometry"

%
% section: Preliminaries
%

\section{Preliminaries}
\label{sec:Preliminaries}

In this section we recall the notions of (curved) $L_\infty$-algebras,
$L_\infty$-morphisms and their twists by Maurer--Cartan elements to
fix the notation. Proofs and further details can be found in
\cite{dolgushev:2005a,dolgushev:2005b,esposito.dekleijn:2018a:pre}.

We denote by $V^\bullet$ a graded vector space over a field $\mathbb{K}$ of
characteristic $0$ and define the \emph{shifted} vector space
$V[k]^\bullet$ by
\begin{equation*}
  V[k]^\ell
  =
  V^{\ell+k}.
\end{equation*}
A degree $+1$ coderivation $Q$ on the coaugmented counital conilpotent
cocommutative coalgebra $S^c(\mathfrak{L})$ cofreely cogenerated by
the graded vector space $\mathfrak{L}[1]^\bullet$ over $\mathbb{K}$ is
called an \emph{$L_\infty$-structure} on the graded vector space
$\mathfrak{L}$ if $Q^2=0$. The (universal) coalgebra
$S^c(\mathfrak{L})$ can be realized as the symmetrized
deconcatenation coproduct on the space
$\bigoplus_{n\geq0}\Sym^n\mathfrak{L}[1]$ where $\Sym^n\mathfrak{L}[1]$ 
is the space of coinvariants for the
usual (graded) action of $S_n$ (the symmetric group in $n$ letters) on
$\otimes^n(\mathfrak{L}[1])$, see e.g. \cite{esposito.dekleijn:2018a:pre}. 
Any degree $+1$ coderivation $Q$ on $S^c(\mathfrak{L})$ is uniquely determined by the
components
\begin{equation}
  Q_n\colon \Sym^n(\mathfrak{L}[1])\longrightarrow \mathfrak{L}[2]
\end{equation}
through the formula 
\begin{equation}
  Q(\gamma_1\vee\ldots\vee\gamma_n)
  =
  \sum_{k=0}^n\sum_{\sigma\in\mbox{\tiny Sh($k$,$n-k$)}}
  \epsilon(\sigma)Q_k(\gamma_{\sigma(1)}\vee\ldots\vee
  \gamma_{\sigma(k)})\vee\gamma_{\sigma(k+1)}\vee
  \ldots\vee\gamma_{\sigma(n)}.
\end{equation} 
Here Sh($k$,$n-k$) denotes the set of $(k, n-k)$ shuffles in $S_n$,
$\epsilon(\sigma)=\epsilon(\sigma,\gamma_1,\ldots,\gamma_n)$ is a sign
given by the rule $
\gamma_{\sigma(1)}\vee\ldots\vee\gamma_{\sigma(n)}=
\epsilon(\sigma)\gamma_1\vee\ldots\vee\gamma_n $ and we use the
conventions that Sh($n$,$0$)=Sh($0$,$n$)$=\{\id\}$ and that the empty
product equals the unit. Note in particular that we also consider a
term $Q_0$ and thus we are actually considering curved
$L_\infty$-algebras (which will be convenient in the following).
Sometimes we also write $Q_k = Q_k^1$ and following
\cite{canonaco:1999a} we denote by $Q_n^i$ the component of $Q_n^i
\colon \Sym^n L[1] \rightarrow \Sym^i L[2]$ of $Q$. It is given by
\begin{equation}
  \label{eq:Qniformula}
  Q_n^i(x_1\vee \cdots \vee x_n)
	=
	\sum_{\sigma \in \mathrm{Sh}(n+1-i,i-1)}
	\epsilon(\sigma) Q_{n+1-i}^1(x_{\sigma(1)}\vee \cdots\vee  x_{\sigma(n+1-i)})\vee
	x_{\sigma(n+2-i)} \vee \cdots \vee x_{\sigma(n)},
\end{equation}
where $Q_{n+1-i}^1$ are the usual structure maps. 
\begin{example}[Curved Lie algebra]
  \label{ex:curvedlie}
  A basic example of an $L_\infty$-algebra is that of a (curved) Lie
  algebra $(\mathfrak{L},R,\D,[\argument,\argument])$ by setting
  $Q_0(1)={ -}R$, $Q_1={ -}\D$,
  $Q_2(\gamma\vee\mu)={ -(-1)^{|\gamma|}}[\gamma,\mu]$ and $Q_i=0$ for
  all $i\geq 3$.   Note that we denoted the degree in $\mathfrak{L}[1]$ by $|\cdot |$. 
\end{example}

Let us consider two $L_\infty$-algebras $(\mathfrak{L},Q)$ and
$(\widetilde{\mathfrak{L}},\widetilde{Q})$.  A degree $0$ counital
coalgebra morphism
\begin{equation*}
  F\colon 
  S^c(\mathfrak{L})
  \longrightarrow 
  S^c(\widetilde{\mathfrak{L}})
\end{equation*}
such that $FQ = \widetilde{Q}F$ is said to be an
\emph{$L_\infty$-morphism}.
A coalgebra morphism $F$ from $S^c(\mathfrak{L})$ to
$S^c(\widetilde{\mathfrak{L}})$ such that $F(1)=1$ is uniquely determined by its
components (also called \emph{Taylor coefficients})
\begin{equation*}
  F_n\colon \Sym^n(\mathfrak{L}[1])\longrightarrow \widetilde{\mathfrak{L}}[1],
\end{equation*}
where $n\geq 1$. Namely, we set $F(1)=1$ and use the formula
\begin{equation*}
  F(\gamma_1\vee\ldots\vee\gamma_n)=
\end{equation*}
\begin{equation}
\label{coalgebramorphism}
  \sum_{p\geq1}\sum_{\substack{k_1,\ldots, k_p\geq1\\k_1+\ldots+k_p=n}}
  \sum_{\sigma\in \mbox{\tiny Sh($k_1$,..., $k_p$)}}\frac{\epsilon(\sigma)}{p!}
  F_{k_1}(\gamma_{\sigma(1)}\vee\ldots\vee\gamma_{\sigma(k_1)})\vee\ldots\vee 
  F_{k_p}(\gamma_{\sigma(n-k_p+1)}\vee\ldots\vee\gamma_{\sigma(n)}),
\end{equation}
where Sh($k_1$,...,$k_p$) denotes the set of $(k_1,\ldots,
k_p)$-shuffles in $S_n$ (again we set Sh($n$)$=\{\id\}$).
We also write $F_k = F_k^1$ and similarly to \eqref{eq:Qniformula} we get 
coefficients $F_n^j \colon \Sym^n L[1] \rightarrow \Sym^j L'[1]$ of $F$ by 
taking the corresponding terms in \cite[Equation~(2.15)]{dolgushev:2006a}. 
Note that $F_n^j$ depends only on $F_k^1 = F_k$ for $k\leq n-j+1$. 
Given an $L_\infty$-morphism $F$ of (non-curved) $L_\infty$-algebras $(\mathfrak{L},Q)$ and
$(\widetilde{\mathfrak{L}},\widetilde{Q})$,
we obtain the map of complexes
\begin{equation*}
  F_1\colon (\mathfrak{L},Q_1)\longrightarrow (\widetilde{\mathfrak{L}},\widetilde{Q}_1).
\end{equation*}
In this case the $L_\infty$-morphism $F$ is called an
\emph{$L_\infty$-quasi-isomorphism} if $F_1$ is a
quasi-isomorphism of complexes.
Given a dgla $(\liealg{L}, \D, [\argument,\argument])$ and an element
$\pi\in \mathfrak{L}[1]^0$ we can obtain a curved Lie algebra by
defining a new differential $\D + [\pi,\argument]$ and considering the
curvature $R^\pi=\D\pi+\frac{1}{2}[\pi,\pi]$.
In fact the same procedure can be applied to a curved Lie algebra
$(\mathfrak{L}, R,\D, [\argument,\argument])$ to obtain the twisted
curved Lie algebra $(\mathfrak{L}, R^\pi, \D +
[\pi,\argument],[\argument,\argument])$, where
\begin{equation*}
  R^\pi
  :=
  R+\D\pi + \frac{1}{2}[\pi,\pi].
\end{equation*} 
The element $\pi$ is called a \emph{Maurer--Cartan element} if it
satisfies the equation
\begin{equation*}
  R+\D\pi+\frac{1}{2}[\pi,\pi]
  =
  0.
\end{equation*}
Finally, it is important to recall that given a dgla morphism, or more
generally an $L_\infty$-morphism, $F\colon \mathfrak{L}\rightarrow
\widetilde{\mathfrak{L}}$, one may associate to any Maurer--Cartan
element $\pi\in\mathfrak{L}[1]^0$ a Maurer--Cartan element
\[\pi_F:=\sum_{n\geq 1} \frac{1}{n!}   F_n(\pi\vee\ldots\vee\pi)\in \widetilde{\mathfrak{L}}[1]^0\]

In order to make sense of these infinite sums we consider complete
filtered $L_\infty$-algebras and we demand that Maurer--Cartan
elements are in a positive filtration, see
\cite{esposito.dekleijn:2018a:pre,dolgushev:2005a} for details on such
filtrations.

%
%An Explicit Formula for the $L_\infty$-Projection
%

\section{An Explicit Formula for the $L_\infty$-Projection and -Inclusion}
\label{sec:ExplicitFormulasDefRetracts}

From the general theory of $L_\infty$-algebras one knows that
$L_\infty$-quasi-isomorphisms always admit
$L_\infty$-quasi-inverses. Moreover, it is well-known that given a
homotopy retract one can transfer $L_\infty$-structures. Explicitly,
given two cochain complexes $(A,\D_A)$ and $(B,\D_B)$ with
\begin{equation}
  \begin{tikzcd} 
	(A,\D_A)
	\arrow[rr, shift left, "i"] 
  &&   
  \left(B, \D_B\right)
  \arrow[ll, shift left, "p"]
	\arrow[loop right, "h"] 
\end{tikzcd} 
\end{equation}
where $h \circ \D_B + \D_B \circ h = \id - i \circ p$ and where $i$ is
a quasi-isomorphism, the homotopy transfer Theorem
\cite[Section~10.3]{loday.vallette:2012a} states that if there exists
an $L_\infty$-structure on $B$, then one can transfer it to $A$ in
such a way that $i$ extends to an $L_\infty$-quasi-isomorphism.

Let us consider the special case of deformation retracts for
DGLA's. More explicitly, let $A,B$ be two DGLA's. A deformation
retract of $(A, \D_A)$ is given by the diagram
\begin{equation}
  \label{eq:Contraction}
  \begin{tikzcd} 
	(A,\D_A)
	\arrow[rr, shift left, "i"] 
  &&   
  \left(B, \D_B\right)
  \arrow[ll, shift left, "p"]
	\arrow[loop right, "h"] 
\end{tikzcd} 
\end{equation}
where $i$ and $p$ are quasi-isomorphisms of cochain complexes with
homotopy $h$, i.e. $h\D_B + \D_B h = \id_B - ip$, as well as
\begin{equation*}
  p\circ i
  =
  \id_A,
  \quad \quad 
  h^2 
  =
  0,
  \quad \quad 
  h \circ i
  =
  0
  \quad \text{ and } \quad 
  p\circ h 
  = 
  0.
\end{equation*} 
In addition, we assume that $i$ is a DGLA morphism. As already
mentioned, the homotopy transfer theorem and the invertibility of
$L_\infty$-quasi-isomorphisms imply that $p$ extends to an
$L_\infty$-quasi-isomorphism denoted by $P$, see
e.g. \cite[Prop.~10.3.9]{loday.vallette:2012a}.  In the following we
give a more explicit description of $P$.
The DGLA structures yield the codifferentials $Q_A$ on $\Sym (A[1])$
and $Q_B$ on $\Sym(B[1])$ and the map $h$ extends to a homotopy $H_n
\colon \Sym^n (B[1]) \rightarrow \Sym^n(B[1])[-1]$ with respect to
$Q_{B,n}^n \colon \Sym^n (B[1]) \rightarrow \Sym^n(B[1])[1]$, see
e.g. \cite[p.~383]{loday.vallette:2012a} for the construction on the
tensor algebra, which adapted to our setting works roughly like:
we define the operator 
	\begin{align*}
	K_n
	\colon 
	\Sym^n(B[1])
	\to 
	\Sym^n(B[1])
	\end{align*}
by 
	\begin{align*}
	K_n(x_1\vee\cdots\vee x_n)
	=
	\frac{1}{n!} 
	\sum_{i=0}^{n-1}
	\sum_{\sigma\in S_n}
	\frac{\epsilon(\sigma)}{n-i}ipX_{\sigma(1)}\vee\cdots\vee ipX_{\sigma(i)}
	\vee X_{\sigma(i+1)}\vee X_{\sigma(n)}.
	\end{align*}
Note that here we sum over the whole symmetric group and 
not the shuffles, since in this case the formulas are easier. We extend $-h$ to a 
coderivation to  $\Sym(B[1])$, i.e.
	\begin{align*}
	\tilde{H}_n(x_1\vee\cdots\vee x_n):=
	-\sum_{\sigma\in \mathrm{Sh}(1,n-1)}
	\epsilon(\sigma) \;
	hx_{\sigma(1)}\vee x_{\sigma(2)}\vee\cdots\vee x_{\sigma(n)}
	\end{align*}
and define 
	\begin{align*}
	H_n=K_n\circ \tilde{H}_n
	=
	\tilde{H}_n\circ K_n.
	\end{align*}
Since $i$ and $p$ are chain maps, we have 
	\begin{align*}
	K_n\circ Q_{B,n}^n=Q_{B,n}^n\circ K_n,
	\end{align*}
where $Q_{B,n}^n$ is the extension of the differential $Q_{B,1}^1 = - \D_B$ to 
$\Sym^n(B[1])$ as coderivation. Hence we have 
	\begin{align*}
	Q_{B,n}^n H_n + H_n Q_{B,n}^n
	=
	(n\cdot\id-ip)\circ K_n,
	\end{align*}
where $ip$ is extended as a coderivation to $\Sym(B[1])$. A combinatorial and 
not very enlightning computation shows that finally 
\begin{equation}
	\label{Eq: ExtHomotopy}
  	Q_{B,n}^n H_n + H_n Q_{B,n}^n
	=
	\id - (ip)^{\vee n}.
\end{equation}
Suppose that we have constructed a morphism of coalgebras $P$
with structure maps $P_k^1 \colon \Sym^k(B[1]) \rightarrow A[1]$ that
is an $L_\infty$-morphism up to order $k$, i.e.
\begin{equation*}
  \sum_{\ell=1}^m P^1_\ell \circ Q_{B,m}^\ell
  =
  \sum_{\ell=1}^m Q_{A,\ell}^1\circ P^\ell_{m}
\end{equation*}
for all $m \leq k$. Then we have the following statement.
\begin{lemma}
  Let $P \colon \Sym(B[1]) \rightarrow \Sym (A[1])$ be an
  $L_\infty$-morphism up to order $k\geq 1$. Then
  \begin{equation}
    \label{eq:Linftykplusone}
    L_{\infty,k+1}
    =
    \sum_{\ell = 2}^{k+1} Q_{A,\ell}^1 \circ P^\ell_{k+1} - \sum_{\ell =1}^{k} 
    P_\ell^1 \circ 	Q^\ell_{B,k+1}
		=
		Q^1_{A,2} \circ P^2_{k+1} - P^1_k \circ Q^k_{B,k+1}
  \end{equation}
  satisfies
  \begin{equation}
    \label{eq:linftycommuteswithq}
    L_{\infty,k+1} \circ Q_{B,k+1}^{k+1}
    =
    -Q_{A,1}^1 \circ L_{\infty,k+1}.
  \end{equation}
\end{lemma}
\begin{proof}
  The statement follows from a straightforward computation. For
  convenience we omit the index of the differential:
  \begin{align*}
    L_{\infty,k+1} Q_{k+1}^{k+1}
    & = 
    \sum_{\ell = 2}^{k+1} Q_{\ell}^1 (P\circ Q)^\ell_{k+1}  
    - \sum_{\ell = 2}^{k+1} \sum_{i=1}^k Q_{\ell}^1 P^\ell_i Q^i_{k+1}
    + \sum_{\ell =1}^{k}\sum_{i=1}^k P_\ell^1 Q^\ell_i Q_{k+1}^i  \\
    & =
    \sum_{\ell = 2}^{k+1} Q_{\ell}^1 (Q\circ P)^\ell_{k+1}  
    - \sum_{\ell = 2}^{k+1} \sum_{i=1}^k Q_{\ell}^1 P^\ell_i Q^i_{k+1}
    + \sum_{\ell =1}^{k}\sum_{i=1}^k Q_\ell^1 P^\ell_i Q_{k+1}^i  \\
    & =
    -Q_1^1 (Q \circ P)^1_{k+1} + Q_1^1 \sum_{i=1}^k P^1_{i}Q^i_{k+1} 
    = 
    - Q_1^1 L_{\infty,k+1},
  \end{align*}
  where the last equality follows from $Q_1^1Q_1^1 = 0$.
\end{proof}
This allows us to obtain the $L_\infty$-quasi-inverse of $i$, denoted
by $P$, in \eqref{eq:Contraction} recursively:
\begin{proposition}
  \label{prop:Infinityprojection}
  Defining $P_1^1 = p$ and $P_{k+1}^1 = L_{\infty,k+1} \circ H_{k+1}$
  for $k \geq 1$ yields an $L_\infty$-quasi-isomorphism $P \colon
  \Sym(B[1]) \rightarrow \Sym (A[1])$ that is quasi-inverse to $i$.
\end{proposition}
\begin{proof}
  We observe $P_{k+1}^1(ix_1 \vee \cdots \vee i x_{k+1}) = 0$ for all
  $k\geq 1$ and $x_i \in A$, which directly follows from $h\circ i =
  0$ and thus $H_{k+1} \circ i^{\vee (k+1)} = 0$.  In addition, one
  also has for all $k\geq 1$ the identity
  $L_{\infty,k+1}(ix_1,\dots,ix_{k+1}) = 0$, which follows from the
  definition of $L_{\infty,k+1}$ and the fact that $i$ is a morphism
  of DGLAs. We know that $P$ is an $L_\infty$-morphism up to order
  one. Suppose that we already know that it is an $L_\infty$-morphism
  up to order $k \geq 1$, then this implies
  \begin{align*}
    P_{k+1}^1 \circ Q^{k+1}_{k+1}
    & = 
    L_{\infty,k+1} \circ H_{k+1} \circ Q_{k+1}^{k+1} \\
    &=
    L_{\infty,k+1} - L_{\infty,k+1} \circ Q_{k+1}^{k+1}\circ H_{k+1} 
    -L_{\infty,k+1} \circ (i\circ p)^{\vee (k+1)} \\
    & =
    L_{\infty,k+1} + Q_1^1 \circ P_{k+1}^1 
  \end{align*}
  by the above lemma, and therefore
  \begin{equation*}
    P_{k+1}^1 \circ Q^{k+1}_{k+1} -  Q_1^1 \circ P_{k+1}^1
    =
    L_{\infty,k+1}. 
  \end{equation*}
  Hence $P$ is an $L_\infty$-morphism up to order $k+1$ and the
  statement follows inductively.
\end{proof}

%\subsection{$L_\infty$-Inclusion for Deformation Retracts of DGLA's}
%\label{sec:LinftyInclusion}
%Let us again assume we have a deformation retract as in \eqref{eq:Contraction}.
%However, 
Let us now we assume that $p\colon B\to A$ in the deformation retract
\eqref{eq:Contraction}is a DGLA morphism and that $i$ is just a chain
map. Then we can analogously give a formula for the extension $I$ of
$i$ to an $L_\infty$-quasi-isomorphism.

\begin{proposition}
  \label{prop:Infinityinclusion}
  The coalgebra map $I\colon \Sym^\bullet (A[1])\to \Sym^\bullet
  (B[1])$ recursively defined by the maps $I_1^1=i$ and $I_k^1=h\circ
  L_{\infty,k}$ for $k\geq 2$ is an $L_\infty$-quasi inverse of $p$.
  Since $h^2=0= h\circ i$, one even has $I_k^1 = h \circ Q^1_{2} \circ
  I^2_{k}$.
\end{proposition}
\begin{proof}
	We proceed by induction: assume that $I$ is an $L_\infty$-morphism up to 
	order $k$, then we have 
	\begin{align*}
		I^1_{k+1}Q_{A,k+1}^{k+1} - Q_{B,1}^1I^1_{k+1}&= 
		-Q_{B,1}^1\circ h\circ L_{\infty,{k+1}}
		+h\circ L_{\infty,{k+1}}\circ Q_{A,k+1}^{k+1}\\&
		=-Q_{B,1}^1\circ h\circ L_{\infty,{k+1}}-h\circ Q_{B,1}^{1}\circ L_{\infty,{k+1}}\\&
		=(\id-i\circ p)L_{\infty,{k+1}}.
	\end{align*}
	We used that $Q_{B,1}^1=-\D_B$ and the homotopy equation of $h$. 
	Moreover, since $p$ is a DGLA morphism and $p\circ h=0$, we have that 
	$p\circ L_{\infty,{k+1}}=0$ for $k\geq 0$. Since $I$ is an $L_\infty$-morphism 
	up to order one, i.e. a chain map, the claim is proven. 
\end{proof}

%
% section: taylor series of multivector fields
%
\section{Reduction of Multivector Fields}
\label{sec:TaylorSeriesofTpoly} 

In the following, we want to use the above language and considerations
to formulate a reduction scheme for multivector fields.  We first
introduce a new complex which of multivector fields which contains the
data of Hamiltonian actions in the case of Lie group actions
$\Phi\colon G\times M\to M$.
\begin{definition}[Equivariant Multivectors]
  The DGLA of equivariant multivector fields is given by the complex $
  T^\bullet_\liealg{g} (M)$ defined by
  \begin{align*}
    T_\liealg{g}^k (M)
    =
    \bigoplus_{2i+j=k} (\Sym^i\liealg{g}^* \tensor \Secinfty(\Anti^{j+1}TM) )^G
    =
    \bigoplus_{2i+j=k} (\Sym^i\liealg{g}^* \tensor T_{\mathrm{poly}}^j(M) )^G,
  \end{align*}
  together with the trivial differential and the following Lie bracket
  \begin{align*}
    [\alpha\tensor X, \beta\tensor Y]_{\liealg{g}}
    =
    \alpha\vee \beta\tensor[]   [X,Y]
  \end{align*}
  for any $\alpha\tensor X, \beta\tensor Y \in  T_\liealg{g}^\bullet(M)$.
\end{definition}
Here $[\argument,\argument]$ refers to the usual
Schouten--Nijenhuis bracket on $T_{\mathrm{poly}}(M)$.
Notice that invariance with respect to the group action means
invariance under the transformations $\Ad_g^*\tensor \Phi_g^*$ for all
$g\in G$.
We can equivalently interpret this complex in terms of polynomial maps
$\liealg{g}\to T_{\mathrm{poly}}^j(M)$ which are equivariant with
respect to adjoint and push-forward action. Using this point of view,
the bracket can be rewritten as
\begin{align}
  [X,Y]_{\liealg{g}}(\xi)
  =
  [X(\xi),Y(\xi)].
\end{align}
Furthermore, we introduce the canonical linear map
\begin{align}
  \lambda
  \colon 
  \liealg{g} \ni \xi \longmapsto \xi_M\in T^0_{\mathrm{poly}}M,
\end{align}	 
where $\xi_M$ denotes the fundamental vector field corresponding to
the action $\Phi$.  It is easy to see that $\lambda $ is central and
as a consequence we can turn $ T^\bullet_\liealg{g} M$ into a
\emph{curved} Lie algebra with curvature $\lambda$.
Now let $(M, \pi)$ be a Poisson manifold and denote by $\{ \argument,
\argument \}$ the corresponding Poisson bracket. Recall that an
(equivariant) momentum map for the action $\Phi$ is a map $J\colon
\liealg{g} \to \Cinfty(M)$ such that
\begin{equation}
  \label{eq:momentummap}
  \xi_M = \{ \argument, J_\xi \}
  \quad
  \text{and}
  \quad
  J_{[\xi, \eta]} 
  =
  \{ J_\xi,  J_\eta \}.
\end{equation}
An action $\Phi$ admitting a momentum map is what we called
\emph{Hamiltonian}.  In the following we prove a characterization of
Hamiltonian actions in terms of equivariant multivectors.
\begin{lemma}
  \label{lem:MChamiltonian}
  The curved Maurer--Cartan elements of $ T_\liealg{g}^\bullet (M)$
  are equivalent to pairs $(\pi,J)$, where $\pi$ is a $G$-invariant
  Poisson structure $J$ is a momentum map $J\colon \liealg{g} \to
  T^{-1}_{\mathrm{poly}}(M)$.
\end{lemma}
\begin{proof}
  The curved Maurer--Cartan equation reads
  \begin{align*}
    \lambda +\frac{1}{2}[\Pi,\Pi]_\liealg{g}=0
  \end{align*}
  for $\Pi\in T_\liealg{g}^1 (M)$. If we decompose $\Pi=\pi-J\in
  (T_{\mathrm{poly}}^1(M) )^G\oplus (\liealg{g}^* \tensor
  T_{\mathrm{poly}}^{-1}(M) )^G $, it is easy to see that the curved
  Maurer--Cartan equation together with the invariance of the elements
  is equivalent to the conditions \eqref{eq:momentummap} defining the
  momentum map.
\end{proof}
As in the Marsden-Weinstein reduction procedure, we fix a
constraint surface $C\subseteq M$, by choosing an equivariant map
$J\colon M\to \liealg{g}^*$ and setting $C=J^{-1}(\{0\})$.  Here we
always assume that $0 \in \liealg{g}^*$ is a regular value of the
momentum map, making $C$ a closed embedded submanifold of M. Note that
$G$ acts canonically of $C$, since $J$ is equivariant.  From now on we
also require the action $\Phi$ to be proper around $C$ and free on
$C$.

To implement this choice in our algebraic setting we consider from now
on the curved differential graded Lie algebra
\begin{align}
  \label{eq:curvedequitpolywithfixedJ}
	(T_\liealg{g}^\bullet (M),\lambda,-[J,\argument],[\argument,\argument]).
\end{align}	  
Note that this is in fact a curved Lie algebra since $[J,J]=0 =
[\lambda,\argument]$.  We have to move to the formal setting in order
to see why this curved Lie algebra is actually interesting. Therefore,
let us consider the curved Lie algebra
\begin{align*}
  (T_\liealg{g}^\bullet (M)[[\hbar]],\hbar\lambda,-[J,\argument],[\argument,\argument]). 
\end{align*}
Note that one advantage of the setting of formal power series is that
we immediately get a complete filtration by the $\hbar$-degrees,
i.e. by setting
\begin{equation*}
  \mathcal{F}^k T_\liealg{g}^\bullet (M)[[\hbar]] 
	= 
	\hbar^k T_\liealg{g}^\bullet (M)[[\hbar]].
\end{equation*}
In particular, if we consider formal Maurer-Cartan elements $\hbar
(\pi- J') \in \hbar T_\liealg{g}^1 (M)[[\hbar]]$, then the twisting
procedures and infinite sums from Section~\ref{sec:Preliminaries} are
all well-defined.
\begin{lemma}
  The formal curved Maurer-Cartan elements of $(T_\liealg{g}^\bullet
  (M)[[\hbar]],\hbar\lambda,-[J,\argument],[\argument,\argument])$ are
  equivalent to pairs $\hbar(\pi,J')$, where $\pi$ is a $G$-invariant
  formal Poisson structure with formal moment map $J+\hbar J'\colon
  \liealg{g} \to T^{-1}_{\mathrm{poly}}(M)[[\hbar]]$.
\end{lemma}
\begin{proof}
  The proof follows directly by Lemma~\ref{lem:MChamiltonian} by
  counting $\hbar$-degrees.
\end{proof}
The rest of this paper is devoted to the construction of a curved
$L_\infty$-morphism
\begin{align*}
	\mathrm{T}_\red
	\colon 
	(T_\liealg{g} (M)[[\hbar]],\hbar\lambda,-[J,\argument],[\argument,\argument])
	\longrightarrow
	(\Tpoly(M_\red)[[\hbar]],0,0,[\argument,\argument])
\end{align*}
with $M_\red:= C/G$. This morphism is frequently referred to as 
\emph{reduction morphism}.

\subsection{Taylor Series Expansion around $C$}
\label{sec:TaylorSeriesofTpoly1}

The main goal of this section is the study of a partial Taylor series
expansion of the multivector field on M around $C$.  Let us assume $M
= C\times \liealg{g}^*$. This is not a strong assumption as we know
from \cite[Lemma~3]{bordemann.herbig.waldmann:2000a} that, if $G$ acts
properly on an open neighbourhood of $C$ we can always find an
$G$-invariant open neighbourhood $M_\nice \subseteq M$ of $C$, such
that there exists a $\group{G}$-equivariant diffeomorphism $M_\nice
\cong U_\nice \subseteq C\times \liealg{g}^*$. Here the Lie group $G$
acts on $C \times \liealg{g}^*$ as
\begin{align*}
	\Phi_g
	= \Phi^C_g\times \Ad^*_{g^{-1}},
\end{align*}
where $\Phi^C$ is the induced action on $C$.  Note that in this
setting the momentum map on $U_\nice$ is simply given by the
projection to $\liealg{g}^*$.
The idea of a Taylor expansion uses the fact that we have the
isomorphism
\begin{align*}
  \Tpoly^k(C\times \liealg{g}^*) \cong
  \bigoplus_{i+j=k}\Cinfty(C\times\liealg{g}^*)
  \tensor_{\Cinfty(C)}(\Anti^i\liealg{g}^*\tensor\Tpoly^j(C)).
\end{align*} 
First, we define 
\begin{align*}
  T_{\liealg{g}^*}\colon \Cinfty(C\times\liealg{g}^*) \ni f \longmapsto
  \sum_{I \in \mathbb{N}_0^n}^\infty \frac{1}{I!}  e_I \otimes
  \iota^*\frac{\del}{\del \alpha_I} f \in
  \prod_i(\Sym^i\liealg{g}\tensor\Cinfty(C)),
\end{align*}
where $\alpha_i e^i$ are coordinates on $\liealg{g}^*$ and $\iota^*$
the restriction to $C$.

\begin{lemma}\label{Lem: InvTayl}
  The map $T_{\liealg{g}^*}$ is equivariant, i.e.
  \begin{align}
    T_{\liealg{g}^*}\circ \Phi_{g,*}
    =
    (\Ad_g\tensor \Phi^C_{g,*})\circ T_{\liealg{g}^*}.
  \end{align}
\end{lemma}
\begin{proof}
  We just observe that
  \begin{align*}
    \Phi_{g}^*\circ \frac{\partial}{\partial \alpha_i}
    =
    (\Ad_{g^{-1}})^i_j \cdot (\Phi^C_g)^*\circ \frac{\partial}{\partial \alpha_j}
  \end{align*}
  for $\Ad_g e_i=(\Ad_g)_i^j e_j$. Hence we have 
  \begin{align*}
    T_{\liealg{g}^*}(\Phi_{g,*}f)
    =
    \sum_{I \in \mathbb{N}_0^n}^\infty \frac{1}{I!}
    e_I \otimes \iota^*\frac{\del}{\del \alpha_I} \Phi_{g,*}f
    =
    (\Ad_g\tensor \Phi^C_{g,*})\circ T_{\liealg{g}^*}f
  \end{align*}
  by shifting the components $(\Ad_{g^{-1}})^i_j = (\Ad_g)^j_i$ to the 
  symmetric powers of $\liealg{g}$.
\end{proof}
\begin{remark}
  It is now clear that this map can be restricted to invariant
  functions in order to obtain invariant elements in
  $\prod_i(\Sym^i\liealg{g}\tensor\Cinfty(C))$. Moreover, with a
  slight adaption of the proof of the Borel-Lemma, see
  e.g. \cite[Theorem~1.3]{moerdijk.reyes:1991a}, one can show that the map
  $T_{\liealg{g}^*}$ is surjective. The more remarkable fact is that
  the properness of the action ensures that the map
  \begin{align*}
    T_{\liealg{g}^*}\colon \Cinfty(M\times\liealg{g}^*)^G
    \ni f
    \mapsto 
    \sum_{I \in \mathbb{N}_0^n}^\infty \frac{1}{I!}
    e_I \otimes \iota^*\frac{\del}{\del \alpha_I} f
    \in \prod_i(\Sym^i\liealg{g}\tensor\Cinfty(C))^G
  \end{align*}
  is surjective. We  omit this proof as we do not use it here and it
  is just an adaption of the corresponding
  statement in \cite{Miaskiwskyi:2019a:pre}.
\end{remark}
We extend this map to $\Tpoly^\bullet(C\times \liealg{g}^*)$ via 
\begin{align*}
  T_{\liealg{g}^*}\colon
  \Tpoly^k(C\times \liealg{g}^*)
  \ni (f\tensor\xi\tensor X)
  \mapsto
  \sum_{I \in \mathbb{N}_0^n}^\infty \frac{1}{I!}
  e_I \otimes\xi \tensor\iota^*\frac{\del}{\del \alpha_I}f\cdot X
  \in \prod_i(\Sym^i\liealg{g}\tensor \Anti\liealg{g}^*\tensor\Tpoly(C))
\end{align*}
and using Lemma \ref{Lem: InvTayl}, we see that also this map can be
restricted to invariant multivector fields:

\begin{definition}[Taylor Expansion around $C$]
  The map 
  \begin{equation}
    \label{eq:TaylorSeries}
    T_{\liealg{g}^*} \colon
    (\Sym\liealg{g}^*\otimes\Tpoly(C\times\liealg{g}^*))^\group{G}
    \longrightarrow 
    T_\Tay (C\times \liealg{g}^*)
    \coloneqq 
    (\Sym\liealg{g}^* \otimes 
    \prod_{i=0}^\infty (\Sym^i \liealg{g} \otimes \Anti \liealg{g}^* 
    \otimes \Tpoly(C)))^\group{G}
  \end{equation}
is called Taylor expansion around $C$.
\end{definition}
Having in mind that the vector space $\prod_{i=0}^\infty (\Sym^i
\liealg{g} \otimes \Anti \liealg{g}^* \otimes \Tpoly(C)))^\group{G}$
is just consisting of Taylor expansions, it is not surprising that it
also inherits the structure of a DGLA: for $P,Q\in \prod_i \Sym^i
\liealg{g}$ and $\xi,\eta\in \liealg{g}^*$, the brackets are given by
\begin{align*}
  [P,Q]
  & =
  0,
  \quad \quad 
        [P\tensor\xi,Q]
        =
        P\vee \inss(\xi)Q,
   \\ 
        [P\tensor\xi,Q\tensor\eta]
 & = 
  P\vee\inss(\xi)Q\tensor \eta-Q\vee\inss(\eta)P\tensor \xi,
\end{align*} 
and they are extended as a Gerstenhaber bracket with respect to the
graded commutative product
\begin{align*}
	(P\tensor\xi)\cdot (Q\tensor\eta)
	:=
	P\vee Q\tensor\xi\wedge\eta.
\end{align*}
We combine it with the usual DGLA structure on $\Tpoly(C)$ and extend
it as in the case of $T_\liealg{g}^\bullet(M)$ trivially to all of
$T_\Tay(C\times \liealg{g}^*)$.  Summarizing, we have a DGLA structure
on the Taylor expansion around $C$ with zero differential.
\begin{lemma}
  The Taylor expansion
  \begin{align}
    T_{\liealg{g}^*}\colon 
    T_\liealg{g}(M)
    \longrightarrow 
    T_\Tay(C\times \liealg{g}^*)
  \end{align}  
  is a DGLA morphism. 
\end{lemma}
\begin{proof}
  This is an easy verification on generators. 
\end{proof}
As a next step we want to include the curvature $\lambda\in T_\liealg{g}^2(M)$
from Section~\ref{sec:TaylorSeriesofTpoly}. Recall that
\begin{align*}
	\lambda
	=
	e^i\tensor (e_i)_M\in T_\liealg{g}^2(M)=(\liealg{g}^*\tensor \Tpoly^{0}(M))^G. 
\end{align*}
Using our assumption that $M=C\times \liealg{g}^*$ and that $G$ acts as
the product of the action on $C$ and the coadjoint action, we see that
\begin{align}
  \label{eq:fundamentalVectorFieldsinCtimesg}
	(e_i)_M
	=
	(e_i)_C+\alpha_k f_{ji}^k\frac{\partial}{\partial \alpha_j},
\end{align}		
where $(e_i)_C$ denotes the fundamental vector field of the action on $C$
and where $f_{ji}^k$ are the structure constants of $\liealg{g}$. 
This means in particular that 
\begin{align*}
	 T_{\liealg{g}^*}(\lambda)
	 =
	 e^i\tensor 1\tensor 1 \tensor(e_i)_C 
	 + 
	 f_{ji}^j e^i\tensor e_k\tensor e^j\tensor 1 \in 
	 T_\Tay(C\times \liealg{g}^*).
\end{align*}
With a slight abuse of notation we write $\lambda$ instead of
$T_{\liealg{g}^*}(\lambda)$.  The same argument leads to the
observation that
\begin{align*}
  T_{\liealg{g}^*}(J)=e^i\tensor e_i\tensor 1\tensor 1,
\end{align*}
where we also write $J$ instead of $T_{\liealg{g}^*}(J)$ in the sequel.

\begin{corollary}
  The map 
  \begin{align}
    T_{\liealg{g}^*}\colon 
    (T_\liealg{g}(M),\lambda,-[J,\argument],[\argument,\argument])
    \longrightarrow 
    (T_\Tay(C\times \liealg{g}^*),\lambda,-[J,\argument],[\argument,\argument])
  \end{align}
  is a morphism of curved Lie algebras. 
\end{corollary}
One main advantage of the Taylor expansion $T_\Tay(C\times
\liealg{g}^*)$ consists in the fact that we have a canonical element
\begin{align*}
	\pi_\KKS
	:= 
	1\tensor\left(\frac{1}{2}f^k_{ij} e_k \otimes e^i\wedge e^j\tensor 1  
	-
	1\tensor e^i\tensor(e_i)_C\right),
\end{align*}
which is not available in $T_\liealg{g}(M)$. Note that $\pi_\KKS$ encodes the action 
on $C$ and the Lie algebra structure on $\liealg{g}$. 
\begin{remark}[Action Lie algebroid]
  \label{remark:actionalgebroid}
	The bundle $C \times \liealg{g} \rightarrow C$ can be equipped with the 
	structure of a Lie algebroid since $\liealg{g}$ acts on $C$ by the 
	fundamental vector fields. The bracket of this \emph{action Lie algebroid} 
	is given by 
	\begin{equation}
	  \label{eq:BracketActionAlg}
	  [\xi, \eta]_{C\times\liealg{g}}(p) 
		=
		[\xi(p),\eta(p)] - (\Lie_{\xi_C}\eta)(p) + 
		(\Lie_{\eta_C}\xi)(p)
	\end{equation}
	for $\xi,\eta \in \Cinfty(C,\liealg{g})$. The anchor is given by 
	$\rho(p,\xi) = -\xi_C\at{p}$. In particular, one can check that 
	$\pi_\KKS$ is the negative of the linear Poisson structure on its dual 
	$C \times \liealg{g}^*$ in the convention of \cite{neumaier.waldmann:2009a}.
\end{remark}
The canonical $\pi_\KKS$ is of big importance since it is part of some
kind of normal form for every invariant Poisson structure on $C\times
\liealg{g}^*$ with moment map $J$. In the Taylor expansion this is
becomes more clear in the following lemma:

\begin{lemma}
  \label{lemma:TaylorexpansionPoissonnormalform}
  Let $\pi \in \big(\prod_{i=0}^\infty(\Sym^i \liealg{g} \otimes \Anti
  \liealg{g}^* \otimes \Tpoly(C))\big)^\group{G} \subseteq
  T_\Tay(C\times \liealg{g}^*)$ be a curved Maurer--Cartan element,
  then
  \begin{equation}
    \pi
    =
    \pi_\KKS+ \pi_C
  \end{equation}
  with $\pi_C
  \in (\prod_{i=0}^\infty\Sym^i\liealg{g} \otimes
  \Tpoly^1(C))^\group{G}$.
\end{lemma}
\begin{proof}
  By \eqref{eq:fundamentalVectorFieldsinCtimesg} we have for
  $\xi\in\liealg{g}$, $c\in C$ and $\alpha = \alpha_i e^i\in
  \liealg{g}^*$
  \begin{equation*}
    \xi_M \at{(c,\alpha)}
    =
    - (\ins(\D J(\xi))\pi)\at{(c,\alpha)}
    =
    \xi_C\at{c} + \xi_{\liealg{g}^*}\at{\alpha}
    =
    \xi_C\at{c} - f^i_{jl}\ins(e_i)\alpha \; e^j(\xi) \frac{\del}{\del \alpha_l}.
  \end{equation*}
  This implies directly	
  \begin{equation*}
    \pi
    =
    \pi_C + (e_i)_C \wedge \frac{\del}{\del \alpha_i} + 
    \frac{1}{2} \alpha_k f^k_{ij}\frac{\del}{\del \alpha_i}\wedge 
    \frac{\del}{\del \alpha_j},
  \end{equation*}
  where $\pi_C \in \Secinfty(M,\Anti^2 TC)$ is tangent to $C$, but can
  possibly depend on all of $M = C \times \liealg{g}^*$. In the
  Taylor expansion $\frac{\del}{\del \alpha_l}$ corresponds to
  $\ins(e^l)$ and the lemma is shown.
\end{proof}
Comparing now the terms in $[\pi,\pi]=0$ with same $\liealg{g}^*$ and
$C$ degrees gives hints concerning the coefficient function of $\pi_C$
that can also depend on $\liealg{g}^*$.  In particular, the terms in
$\Secinfty(\Anti^3TC)$ are given by
\begin{equation}
  \label{eq:EquforPiC}
  [\pi_C,\pi_C] + 2 (e_i)_C \wedge \left[\frac{\del}{\del \alpha_i}, \pi_C\right]
	=
	0.
\end{equation}
To conclude this section, we define for later use the operator
\begin{equation}
  \label{eq:DifferentialDel}
	\del 
	:=
	\id \otimes \inss(e^i) \otimes \id \otimes (e_i)_C\wedge.
\end{equation}
Note that we assume the Koszul sign rule, i.e. applying $\del$ to $\xi
\otimes P \otimes \alpha \otimes X$ we get a sign
$(-1)^{\abs{\alpha}}$.  We directly see that $\del^2 = 0$ and Equation~\eqref{eq:EquforPiC} can be written as
\begin{equation}
  \label{eq:MCforPiC}
	\frac{1}{2} [\pi_C,\pi_C] + (e_i)_C \wedge \inss(e^i)\pi_C
	=
	\frac{1}{2} [\pi_C,\pi_C] + \del \pi_C
	=
	0.
\end{equation}
\subsection{The Cartan Model of Multivector Fields}

In the case of symplectic manifolds, it has been shown in \cite{reichert:2017a} 
that quantization and reduction commute by exploiting the following diagram  
\begin{align*}
	((\Sym\liealg{g}^*\tensor\Omega(M))^G,\D_{\liealg{g}})
	\overset{\iota^*}{\longrightarrow}
	((\Sym\liealg{g}^*\tensor\Omega(C))^G,\D_{\liealg{g}})
	\overset{p^*}{\longleftarrow}
	(\Omega^\bullet(M_\red),\D)
\end{align*}
for $M\overset{\iota}{\longleftarrow} C \overset{p}\to M_\red $. Here $p^*$ is a 
quasi-isomorphism and 
$(\Sym\liealg{g}^*\tensor\Omega(C))^G$ is the so-called Cartan model 
for equivariant de Rham cohomology \cite{guillemin.sternberg:1999a}. We aim 
to generalize this result to the setting of Poisson manifolds by using the above observation 
as a guideline. 
For this reason we introduce our notion for the
\emph{Cartan Model of equivariant multivector fields} and compute
its relation with $T_\poly(M_\red)$ and with the Taylor expansion of
multivector fields around $C$ from the previous section.
We start with the following observation:
\begin{proposition}
  \label{prop:cohomologytrivialbracketj}
    The cohomology of the DGLA
    $(T_\Tay(C\times\liealg{g}^*),-[J,\argument],[\argument,\argument])$
    is given by the Lie algebra $\left((\prod_{i=0}^\infty
    (\Sym^i\liealg{g} \otimes\Tpoly(C)))^\group{G},
    [\argument,\argument]\right)$. Therefore, the canonical inclusion 
		\begin{align}
	  \iota
	  \colon 
	  \left(\left(\prod_{i=0}^\infty(\Sym^i\liealg{g} \otimes\Tpoly(C))\right)^\group{G},0, 
		[\argument,\argument]\right)
	  \longrightarrow
	  (T_\Tay(C\times\liealg{g}^*),[-J,\argument],[\argument,\argument]) 
	\end{align}
  becomes a quasi-isomorphism of DGLA's.
\end{proposition}
\begin{proof}
  The map $ h = \inss(e_l) \otimes \id \otimes e^l \wedge \otimes
  \id$ satisfies
  \begin{equation*}
    -[J,\argument] \circ h(\xi \otimes P \otimes \alpha \otimes X) 
    -  h \circ [J,\argument] (\xi \otimes P \otimes \alpha \otimes X) 
    =
    (\deg(\alpha) + \deg(\xi))  (\xi \otimes P \otimes \alpha \otimes X) 
  \end{equation*}
  and the statement follows.
\end{proof}
Note that the cohomology 
$(\prod_{i=0}^\infty (\Sym^i\liealg{g} \otimes\Tpoly(C)))^\group{G}$ can be 
equipped with a non-trivial, but canonical differential.

\begin{proposition}
  \label{prop:delonequidgla}
  The differential $\del$ defined in \eqref{eq:DifferentialDel} turns
  $((\prod_{i=0}^\infty (\Sym^i\liealg{g}
  \otimes\Tpoly(C)))^\group{G},[\argument,\argument])$ into a DGLA.
\end{proposition}
\begin{proof}
  A straightforward computation shows
  \begin{align*}
    \del[\xi \otimes X, \eta \otimes Y]
    &=
    \inss(e^i)(\xi \vee \eta) \otimes (e_i)_C \wedge [X,Y],
    \\
    [\del(\xi\otimes X), \eta \otimes Y]
    &=
    (-1)^k \inss(e^i)(\xi) \vee \eta \otimes X \wedge [ (e_i)_C, Y]
    + \inss(e^i)(\xi)\vee \eta \otimes (e_i)_C \wedge[X,Y],
    \\
      [\xi \otimes X, \del(\eta \otimes Y)]
      &=
      (-1)^{k-1} \xi \vee \inss(e^i)(\eta) \otimes (e_i)_C \wedge [X,Y] 
      - \xi \vee \inss(e^i)(\eta)\otimes [(e_i)_C,X] \wedge Y,
  \end{align*}
  where $X \in \Tpoly^{k-1}(C)$. Using the $G$-invariance we get
  \begin{equation*}
    \xi \otimes [(e_i)_C,X]
    =
    - f_{ij}^k e_k \vee \inss(e^j)\xi \otimes X
    \quad
    \text{   and   }
    \quad
    \eta \otimes [(e_i)_C,Y]
    =
    - f_{ij}^k e_k \vee \inss(e^j)\eta \otimes Y.
  \end{equation*}
  Summarizing, this yields
  \begin{equation*}
    \del[\xi \otimes X, \eta \otimes Y]
    =
    [\del(\xi \otimes X), \eta \otimes Y]
    +(-1)^{k-1} [\xi \otimes X, \del(\eta \otimes Y)]
  \end{equation*}
  and the proposition is shown.
\end{proof}
This motivates the following definition. 
\begin{definition}[Cartan model]
  Let $G$ be a Lie group action on a manifold $C$. The DGLA defined by
  \begin{align}
    \left(\prod_{i=0}^\infty (\Sym^i\liealg{g}
    \otimes\Tpoly(C)))^\group{G},\del,[\argument,\argument]\right)
  \end{align}
  is called \emph{Cartan model} and is denoted by $T_\Cart(C)$.
\end{definition}
Seen as a module, $\left(\prod_{i=0}^\infty (\Sym^i \liealg{g} \otimes
\Tpoly(C))\right)^\group{G}$ is the dual of the Cartan model
$(\Sym\liealg{g}^*\otimes \Omega(C))^\group{G}$ for the equivariant de
Rham cohomology \cite{guillemin.sternberg:1999a,reichert:2017b}. Even
the differential $\del$ is dual to the insertion $\ins_\bullet = e^i
\vee \otimes \insa((e_i)_C)$ that forms together with the de Rham
differential the coboundary operator in the usual Cartan model for
equivariant cohomology.  In the case of forms the equivariant
cohomology of the principal fiber bundle $C$ is isomorphic to the de
Rham cohomology of the reduced manifold, whereas in our setting we
want to show that we get the multivector fields on $M_\red$ as
cohomology.  Note that we have a canonical DGLA map
\begin{align*}
  p\colon  
  (T_\Cart(C),\del,[\argument,\argument])
  \longrightarrow
  (T_\poly(M_\red),0,[\argument,\argument]),
\end{align*}
which is just given by the projection to the symmetric degree $0$
followed by the projection to $M_\red$. It is well-defined since
invariant multivector fields are projectable.
\begin{proposition}
  \label{prop:redfromctomred}
  The DGLA-map  
  \begin{align}
    p\colon 
    \left(\left(\prod_{i=0}^\infty (\Sym^i\liealg{g}
    \otimes\Tpoly(C))\right)^\group{G},\del, [\argument,\argument]\right)
    \longrightarrow
    (\Tpoly(M_\red),0,[\argument,\argument])
  \end{align}  
  is a quasi-isomorphism. 
\end{proposition}
\begin{proof}
  Consider the principal bundle $\pr \colon C \rightarrow M_\red$ and
  choose a principal bundle connection $\omega =\omega^i \otimes e_i
  \in \Omega^1(C) \otimes \liealg{g}$, i.e. an equivariant horizontal
  lift inducing
  \begin{equation*}
    TC 
    = 
    \Ver(C) \oplus \Hor(C) 
    =
    \ker T\pr \oplus \ker \omega ,
  \end{equation*}
  where $\ker \omega \cong \pr^* TM_\red$. Then we can construct a homotopy for
  $\del$ by $h = e_i \vee \otimes \ins(\omega^i)$.  Since
  $\omega^i((e_j)_M) = \delta^i_j$, it satisfies
  \begin{equation*}
    h \del + \del h 
    = 
    (\deg_\liealg{g} + \deg_\ver)\id.
  \end{equation*}
With the vertical degree we mean the degree in the splitting 
$	\Anti^kTC = \bigoplus_{i+j=k}\Anti^{i} \Ver(C)\tensor\Anti^j \Hor(C)$.
\end{proof}
In other words, the above proposition yields for every principal connection
$\omega\in \Omega^1(C)\tensor\liealg{g}$ the following deformation
retract 
\begin{equation}\label{Eq: DefRetrCart}
  \begin{tikzcd} 
    \Tpoly(M_\red)
    \arrow[rr, shift left, "i"] 
    &&   
    (\prod_{i=0}^\infty (\Sym^i\liealg{g}	\otimes\Tpoly(C)))^\group{G}
    \arrow[ll, shift left, "p"]
    \arrow[loop right, "h"] 
  \end{tikzcd} 
\end{equation}
where $i$ denotes the horizontal lift with respect to the connection
$\omega$ and the homotopy $h$ is given on all homogeneous elements by
\begin{align*}
  h ( \xi \otimes X)
  =
  \begin{cases}
    \frac{1}{\deg(\xi) + \deg_\ver(X)} \; e_i \vee \xi \otimes 
    \ins(\omega^i)X   \quad \quad & \text{if} \deg(\xi) + \deg_\ver(X) \neq 0 \\
    0 \quad \quad \quad \quad &\text{else}.
  \end{cases}
\end{align*}
Indeed, the algebraic relations of a deformation retract between $i$,
$p$ and $h$ are easily seen to be verified.
Recall that additionally $p$ is a DGLA morphism, which puts us exactly
in the situation of Proposition~\ref{prop:Infinityinclusion}.  So
before we continue to put the Cartan model in the context of
reduction, we give an explicit formula for a quasi-inverse of $p$ .
	
\begin{proposition}
  \label{prop:iinftycartan}
  For a fixed principal fiber connection $\omega \in
  \Omega^1(C)\otimes \liealg{g}$ with curvature $\Omega\in
  \Omega^2(C)\tensor \liealg g$, one obtains an
  $L_\infty$-quasi-inverse of $p$
  \begin{equation}
    i_\infty 
    \colon 
    \Sym(\Tpoly(M_\red)[1]) 
    \longrightarrow 
    \Sym\left(\left(\prod_{i=0}^\infty (\Sym^i\liealg{g}\otimes\Tpoly(C)
    )\right)^\group{G}[1]\right)
  \end{equation}
 	given by 
 \begin{align}
	i_\infty
	=
	\E^{\Omega}\circ (\cdot)^\hor,
\end{align}
where one extends $(\cdot)^\hor$ as a coalgebra morphism and $\Omega$
as a coderivation of degree $0$. In particular,
\begin{equation}
  i_{\infty,1}(X) 
  =
  X^\hor 
  \quad \text{ and }  \quad
  i_{\infty,2}(X,Y)
  =
  (-1)^{\abs{X}}e_i\tensor \Omega^i(X^\hor,Y^\hor)
\end{equation}
for a basis $\{e_i\}_{i\in I}$ of $\liealg g$.
\end{proposition}
\begin{proof}
  Let us fix a principal connection $\omega\in \Omega^1(C)\tensor
  \liealg g$ and denote by $h$ the corresponding homotopy and by
  $\Omega$ its curvature.  Due to that fact that Equation~\eqref{Eq:
    DefRetrCart} is a deformation retract and $p$ is a DGLA morphism,
  we are exactly in the situation of
  Proposition~\ref{prop:Infinityinclusion} and the statement becomes a
  purely computational issue, so let us start with some
  book-keeping. Throughout the proof, we will make use of the
  following equation for $X\in \Secinfty(\Anti^kTC)$, $Y\in
  \Secinfty(\Anti^\ell TC)$ and $\alpha\in \Omega^1(C)$:
  \begin{align}\label{Eq: DiffGeoEq}
    \tag{$*$}
    \D\alpha(X,Y)
    =
    [\insa(\alpha) X,Y] 
    -(-1)^{k}[X,\insa(\alpha )Y]-
    \insa(\alpha)[X,Y]
  \end{align}
  where for the left-hand side, we define 
  \begin{align*}
    \D\alpha(X,Y) 
    =
    (\D\alpha)_{ij}\insa(\D x^i)X\wedge \insa(\D x^j) Y
  \end{align*}
  in a coordinate patch. The validity of Equation \eqref{Eq:
    DiffGeoEq} for one-forms of the type $\alpha = f \D g$ follows by
  the usual Schouten calculus.  By $\mathbb{R}$-linearity of Equation
  \eqref{Eq: DiffGeoEq}, its validity follows for general 1-forms in
  every coordinate patch and hence also globally.  Let us define,
  using the curvature $\Omega$, the map
  \begin{align*}
    \Omega\colon 
    \Sym^2\big(\prod_i (\Sym^i\liealg g\tensor \Tpoly(C)[1])^G\big)
    \longrightarrow 
    (\prod_i \Sym^i\liealg g\tensor \Tpoly(C))^G[1]
  \end{align*}
  defined on homogeneous and factorizing elements $P_j\tensor X_j\in
  \prod_i (\Sym^i\liealg g\tensor \Tpoly(C)[1])^G$ , $j=1,2$ by
  \begin{align*}
    \Omega(P_1\tensor X_1\vee P_2\tensor X_2)
    =
    (-1)^{\abs{X_1}}e_i\vee P_1\vee P_2\tensor \Omega^i(X_1,X_2).
  \end{align*}
  This map is well defined, i.e. in fact graded symmetric, and of
  degree $0$. With a slight abuse of notation we denote also by
  \begin{align*}
    \Omega\colon 
    \Sym^\bullet\big(\prod_i (\Sym^i\liealg g\tensor \Tpoly(C)[1])^G\big)
    \longrightarrow
    \Sym^{\bullet-1}\big(\prod_i (\Sym^i\liealg g\tensor \Tpoly(C)[1])^G\big)
  \end{align*}	 
  its extension as a coderivation of degree $0$, i.e.
  \begin{align*}
    \Omega(X_1\vee\dots \vee X_k)
    =
    \sum_{\sigma \in \mathrm{Sh}(2,k-2)} \epsilon(\sigma)
    \Omega(X_{\sigma(1)}\vee X_{\sigma(2)})\vee X_{\sigma(3)}\vee\dots\vee X_{\sigma(k)}
  \end{align*}
  for $X_j\in \prod_i (\Sym^i\liealg g\tensor \Tpoly(C)[1])^G$. Note
  that for every $k\in \mathbb{N}$ and $X_j\in \prod_i (\Sym^i\liealg
  g\tensor \Tpoly(C)[1])^G$, we have that
  \begin{align*}
    \Omega^k(X_1\vee\dots \vee X_k)
    =
    0
  \end{align*}
  since $\Omega$ decreases the symmetric degree by one and hence the expression 
  \begin{align*}
    \E^\Omega
    :=
    \sum_k\frac{1}{k!}\Omega^k
  \end{align*}	 
  is a well defined map. Since $\Omega$ is a coderivation of degree
  $0$, it is even a coalgebra morphism. Its components are given by
  \begin{align*}
    (\E^\Omega)_k^\ell
    =
    \frac{1}{(k-\ell)!}\Omega^{k-\ell},
  \end{align*}
  which can be seen again by counting symmetric degrees. This shows in
  particular, that $(\E^\Omega\circ (\cdot)^\hor)_1^1=(\cdot)^\hor$.
  We proceed now inductively, so let us assume that $\E^\Omega\circ
  (\cdot)^\hor$ coincides with $i_\infty$ from Proposition~\ref{prop:Infinityinclusion}
   up to order k. For $X_j\in
  \Tpoly(M_\red)[1]$, $j=1,\dots,k+1$, we have
  \begin{align*}
    & i_{\infty,k+1}(X_1  \vee\dots\vee X_{k+1})
    =
    h\circ Q_2^1 \circ i_{\infty,k+1}^2(X_1\vee\dots\vee X_{k+1})\\
    &	=
    \sum_{j=1}^k \sum_{\sigma \in \mathrm{Sh}(i,k+1-i)}\frac{\epsilon(\sigma)}{2} 
    h\circ Q_2^1\bigg(i_{\infty,j}^1(X_{\sigma(1)}\vee\dots\vee X_{\sigma(j)})\vee
    i_{\infty,k-i+1}^1(X_{\sigma(j+1)}\vee\dots\vee X_{\sigma(k+1)})\bigg)\\
    &	=
    \sum_{j=1}^k \sum_{\sigma \in \mathrm{Sh}(i,k+1-i)}\frac{\epsilon(\sigma)}{2} 
    h\circ Q_2^1 \\
    & \quad \quad \quad \bigg(\frac{\Omega^{j-1}}{(j-1)!}
    (X_{\sigma(1)}^\hor\vee\dots\vee X_{\sigma(j)}^\hor)\vee
    \frac{\Omega^{k-j}}{(k-j)!}(X_{\sigma(j+1)}^\hor\vee\dots\vee X_{\sigma(k+1)}^\hor)\bigg).
  \end{align*}
  Let us now take a look at 
  \begin{align*}
    h& \circ  Q_2^1\bigg(\Omega^{j-1}(X_{\sigma(1)}^\hor\vee\dots\vee X_{\sigma(j)}^\hor)
    \vee \Omega^{k-j}(X_{\sigma(j+1)}^\hor\vee\dots\vee X_{\sigma(k+1)}^\hor)\bigg)\\
    &	=
    (-1)^{1+\sum_{k=1}^j\abs{X_{\sigma(j)}}}h[\Omega^{j-1}(X_{\sigma(1)}^\hor\vee
      \dots\vee X_{\sigma(j)}^\hor),
      \Omega^{k-j}(X_{\sigma(j+1)}^\hor\vee\dots\vee X_{\sigma(k+1)}^\hor)]\\
    & =
    \frac{(-1)^{1+\sum_{k=1}^j\abs{X_{\sigma(j)}}}}{k}e_i\tensor \insa(\omega^i)
	 [\Omega^{j-1}(X_{\sigma(1)}^\hor\vee\dots\vee X_{\sigma(j)}^\hor),
	   \Omega^{k-j}(X_{\sigma(j+1)}^\hor\vee\dots\vee X_{\sigma(k+1)}^\hor)]\\
	 & \overset{\eqref{Eq: DiffGeoEq}}{=}
	 \frac{(-1)^{\sum_{k=1}^j\abs{X_{\sigma(j)}}}}{k}e_i\tensor \D \omega^i\big(
	 (\Omega^{j-1}(X_{\sigma(1)}^\hor\vee\dots\vee X_{\sigma(j)}^\hor),
	 \Omega^{k-j}(X_{\sigma(j+1)}^\hor\vee\dots\vee X_{\sigma(k+1)}^\hor)\big).
  \end{align*}
  The factor $\frac{1}{k}$ appears, since $\Omega^k$ raises the
  symmetric degree in $\liealg g$ by $k$ and hence the commutator has
  $k-1$ symmetric degrees in $\liealg g$ degrees and at most one
  vertical degree, since both of the entries are horizontal
  multivector fields. Moreover, since $\insa(\omega^i)$ annihilates
  the terms which do not have a vertical degree, the formula is
  valid. Note that by definition of the curvature of $\omega$, we have
  $\Omega=\D\omega-\frac{1}{2}[\omega,\omega]$ or for a chosen basis
  $\Omega^i=\D\omega^i-\frac{1}{2}f^i_{kl}\omega^k\wedge\omega^l$. Since
  $\omega^i$ vanishes on horizontal lifts, we can write
  \begin{align*}
    h\circ Q_2^1\bigg(&\Omega^{j-1}(X_{\sigma(1)}\vee\dots\vee X_{\sigma(j)})\vee
    \Omega^{k-j}(X_{\sigma(j+1)}\vee\dots\vee X_{\sigma(k+1)})\bigg)\\
    & =
    \frac{(-1)^{\sum_{k=1}^j\abs{X_{\sigma(j)}}}}{k}e_i\tensor \Omega^i\big(
    (\Omega^{j-1}(X_{\sigma(1)}\vee\dots\vee X_{\sigma(j)}),
    \Omega^{k-j}(X_{\sigma(j+1)}\vee\dots\vee X_{\sigma(k+1)})\big)\\
    & =
    \frac{1}{k}\Omega\big(
    (\Omega^{j-1}(X_{\sigma(1)}\vee\dots\vee X_{\sigma(j)}),
    \Omega^{k-j}(X_{\sigma(j+1)}\vee\dots\vee X_{\sigma(k+1)})\big)
  \end{align*}
  and hence 
  \begin{align*}
    & i_{\infty,k+1}(X_1 \vee\dots\vee X_{k+1})\\
    & =
    \sum_{j=1}^k \sum_{\sigma \in \mathrm{Sh}(i,k+1-i)}\frac{\epsilon(\sigma)}{2k} 
    \Omega\bigg(\frac{\Omega^{j-1}}{(j-1)!}(X_{\sigma(1)}\vee\dots\vee X_{\sigma(j)})\vee
    \frac{\Omega^{k-j}}{(k-j)!}(X_{\sigma(j+1)}\vee\dots\vee X_{\sigma(k+1)})\bigg)\\
    & =
    \frac{1}{k!}\Omega^k(X_1 \vee\dots\vee X_{k+1}).
  \end{align*}
  The last equality follows from the observation that 
  \begin{align*}
    \Omega^k(X_1\vee\dots\vee X_{k+1})&=\Omega(\Omega^{k-1}(X_1\vee\dots\vee X_{k+1}))\\&
    =\Omega\big((k-1)!(\E^{\Omega})^2_{k+1}(X_1\vee\dots\vee X_{k+1})\big)\\&
    =
    \sum_{j=1}^k\sum_{\sigma\in \mathrm{Sh}(i,k+1-1)}
    \frac{\epsilon(\sigma)}{2}\frac{(k-1)!}{(j-1)!(k-j)!} \\&
    \Omega\big( \Omega^{j-1}(X_{\sigma(1)}\vee\dots\vee X_{\sigma(j)})\vee
    \Omega^{k-j}(X_{\sigma(j+1)}\vee\dots\vee X_{\sigma(k+1)})\big),
  \end{align*}
  and the proof is completed.
\end{proof} 
\begin{corollary}\label{Cor: SurjProj}
  The induced map at the level of Maurer-Cartan elements
  \begin{align*}
    p\colon MC(T_\Cart(C))
    \longrightarrow 
    MC(\Tpoly(M_\red))
  \end{align*}
  is surjective. 
\end{corollary}
\begin{proof}
  Let $\pi \in \Tpoly^1(M_\red)$ be a Maurer--Cartan element, i.e. a
  Poisson structure. We define
  \begin{align*}
    \Pi
    =
    \sum_{k\geq 1} \frac{1}{k!} i_{\infty,k}(\pi^{\vee k}).
  \end{align*}	 
  This series actually well-defined in $T_\Cart(C)$, since we have 
  \begin{align*}
    \Pi
    =
    \sum_{k\geq 1} \frac{1}{k!} \frac{1}{(k-1)!}\Omega^{k-1}((\pi^\hor)^{\vee k})
  \end{align*}
  using the explicit for of $i_\infty$ as in
  Propostion~\ref{prop:iinftycartan}. But
  \begin{align*}
    \Omega^{k-1}((\pi^\hor)^{\vee k})\in (\Sym^{k-1}\liealg{g}\tensor\Tpoly^1(C))^\group{G},
  \end{align*}	 
  whence $\Pi\in MC(T_\Cart(C))$ is well-defined. The identity
  $p(\Pi)=\pi$ is then clear using again the explicit form.
\end{proof}

\begin{remark}
  In particular, the above proposition shows not only that if $C$ admits a flat
  connection, then $i_\infty$ has $i_{1} =
  (\argument)^\hor$ as only structure map, but also how to correct 
  the horizontal lift in order to obtain an $L_\infty$-quasi-ismorphism.
\end{remark}

Having seen the importance of the ad-hoc defined differential $\del$ on
$T_\Cart(C)$, we lean now again towards
$T_\Tay(C\times\liealg{g}^*)$ and try to find an extension of the 
differential $-[J,\argument]$ in order to make the inclusion $\iota \colon
T_\Cart(C) \rightarrow T_\Tay(C\times\liealg{g}^*)$ a 
quasi-isomorphism with respect to $\del$. As a first step we have:
\begin{proposition}
  \label{prop:bracketpiminusr}
  The map $[\pi_\KKS , \, \cdot\,]$ is a well-defined differential
  on $T_\Tay(C\times\liealg{g}^*)$ that is explicitly given by
  \begin{equation}
    \label{eq:DiffPiKKSR}
	  [\pi_\KKS , \xi \otimes P \otimes \alpha \otimes X]
	  =
	  \xi \otimes \delta_\CE( P \otimes \alpha \otimes X) 
	  + \del(\xi \otimes P \otimes \alpha \otimes X).
  \end{equation}
	Moreover, the canonical inclusion 
	\begin{align*}
	\iota
	\colon 
	(T_\Cart(C),\del,[\argument,\argument])
	\longrightarrow
	(T_\Tay(C\times\liealg{g}^*),[\pi_\KKS-J,\argument],[\argument,\argument]) 
	\end{align*}
	becomes a DGLA morphism. 
\end{proposition}
\begin{proof}
  Since the bracket does not depend on the $\Sym \liealg{g}^*$-part we
  restrict ourselves to $P \otimes \alpha \otimes X$. Let us compute
  \begin{align*}
    \left[\frac{1}{2} f^k_{ij} e_k \otimes e^i \wedge e^j, 
      P\otimes \alpha \otimes X \right]
    & =
    \frac{1}{2} f^k_{ij} \left( e_k \otimes [e^i \wedge e^j, P \otimes \alpha] 
    \otimes X + [e_k,P\otimes \alpha] \wedge e^i \wedge e^j \otimes X 	\right) \\
    & =
    f^k_{ij} e_k \vee \inss( e^j)P \otimes e^i \wedge \alpha \otimes X 
    -\frac{1}{2} f^k_{ij} P \otimes e^i \wedge e^j \wedge \insa(e_k)\alpha 
    \otimes X
  \end{align*}
  and
  \begin{align*}
    [ - e^i \wedge (e_i)_C, P\otimes \alpha \otimes X]
    & = 
    - P \otimes e^i \wedge \alpha \otimes \Lie_{(e_i)_C}X 
    - (-1)^{\abs{\alpha} + \abs{X}} 
    \inss(e^i)P \otimes \alpha \otimes X \wedge (e_i)_C,
  \end{align*}
  where $\abs{X}$ denotes the multivector field degree and
  $\abs{\alpha}$ the form degree. Putting this together we directly
  get \eqref{eq:DiffPiKKSR}.  Since $\pi_\KKS$ is a Poisson structure, 
	we directly see that it squares to zero. Moreover, 
	$[\pi_\KKS,\argument]$ boils down to $\del$ when restricted to elements 
	in the image of the canonical inclusion $\iota$, i.e. in 
	$(1 \otimes \prod_{i=0}^\infty (\Sym^i\liealg{g} \otimes 1 
	\otimes \Tpoly(C)))^\group{G}$.
\end{proof}
Alternatively, the identity
\begin{equation}
  \label{eq:PiKKSRJ}
  [\pi_\KKS, J]
	=
	\lambda,
\end{equation}
implies that the canonical $\pi_\KKS $ defines a curved Maurer-Cartan element
in the curved DGLA $(T_\Tay(C\times \liealg{g}^*),\lambda,-[J,\argument],
[\argument,\argument])$. Therefore, twisting by $\pi_\KKS$ yields a Lie
algebra differential on $T_\Tay(C\times \liealg{g}^*)$ with curvature
zero. The next step is, of course, to check if $\iota$ is still a quasi-isomorphism.

\begin{proposition}
  \label{prop:InclusionCtoMQuis}
  The inclusion
	\begin{equation}
	 \iota \colon
	(T_\Cart(C),\del,[\argument,\argument])
	\longrightarrow
	(T_\Tay(C\times\liealg{g}^*),[\pi_\KKS-J,\argument],[\argument,\argument])	
	\end{equation}
	is a quasi-isomorphism of DGLAs.
\end{proposition}
\begin{proof}
  Let us compute
  the cohomology of $T_\Tay(C\times \liealg{g}^*)$ by interpreting it
  as a double complex. The two differentials are $[-J,\argument]$ and
  $[\pi_\KKS ,\argument]$ and as bigrading we set
  \begin{equation*}
    C^{p,q}
    =
    (\Sym^q\liealg{g}^* \otimes \prod_{i=0}^\infty 
    (\Sym^i \liealg{g} \otimes (\Anti \liealg{g}^* \otimes 
    \Tpoly(C))^{p-q}))^\group{G}.
  \end{equation*}
  One can directly see that the differentials are compatible with the
  bigrading in the sense that
  \begin{equation*}
    [-J,\argument] \colon
    C^{p,q}
    \longrightarrow
    C^{p,q+1},
    \quad \text{and} \quad
	  [\pi_\KKS,\argument] \colon
	  C^{p,q}
	  \longrightarrow
	C^{p+1,q}.
  \end{equation*}
  By Proposition~\ref{prop:cohomologytrivialbracketj} the cohomology
  of $[-J,\argument]$ is given by $(\prod_{i=0}^\infty
  (\Sym^i\liealg{g} \otimes\Tpoly(C)))^\group{G}$, on which the
  horizontal differential $[\pi_\KKS ,\argument]$ is just $\del$.
  Thus $\iota$ is an isomorphism on the first sheet and thus on the
  cohomology.
\end{proof}

The above results show that the Cartan model is an intertwiner 
of $T_\Tay(C\times\liealg{g}^*)$ and $\Tpoly(M_\red)$, which can be 
summarized in the following diagram. 
	\begin{center}\label{Diag: Cartan}
	\begin{tikzcd}
	(T_\Tay(C\times \liealg{g}^*),[\pi_\KKS-R-J,\argument],[\argument,\argument])\\
	(T_\Cart(C),\del,[\argument,\argument])
	\arrow[u, "\iota"]
	\arrow[d, "p"]\\
	(\Tpoly(M_\red),[\argument,\argument]) .
	\arrow[u,rightsquigarrow, bend left=60, "i_\infty"]
	\end{tikzcd}
	\end{center}
So far we have shown that both $\iota$ and $p$ are DGLA morphisms and 
also quasi-isomorphisms. For convenience, we included the $L_\infty$-quasi-inverse 
$i_\infty$ of $p$. From this diagram and the fact that every 
$L_\infty$-quasi-isomorphism is quasi-invertible, we have the following:

\begin{theorem}
  \label{thm:ClassicalTredTaylor}
  There exists an $L_\infty$-quasi-isomorphism 
  \begin{align*} 
    (T_\Tay(C\times\liealg{g}^*),
	[\pi_\KKS-J,\argument],[\argument,\argument])
	\longrightarrow 
	(T_\poly(M_\red),0,[-,-]).
  \end{align*}
\end{theorem}   
Note that the KKS Poisson structure is not defined on 
$M$, but just in an open 
neighbourhood of $C$. Recall that we aim to find a 
curved $L_\infty$-morphism 
	\begin{align*}
	\mathrm{T}_\red
	\colon 
	(T_\liealg{g} (M),\lambda,-[J,\argument],[\argument,\argument])
	\longrightarrow 
	(\Tpoly(M_\red),0,0,[\argument,\argument])
	\end{align*}
and its formal correspondence. To achieve this we proceed in the following way:
we construct a (non-curved) quasi-inverse of $\iota$ in Diagram~\eqref{Diag: Cartan} 
denoted by $P$ and then twist it by $-\pi_\KKS$ in order to find a curved morphism 
\begin{align*}
P^{-\pi_\KKS}
\colon
(T_\Tay(C\times\liealg{g}^*),\lambda,-[J,\argument],[\argument,\argument])
\longrightarrow
(\Tpoly(M_\red),\lambda_\red,[\pi_{\KKS,\red},\argument],[\argument,\argument])
\end{align*}
for 
	\begin{align*}
	\lambda_\red
	:=
	\sum_{k\geq 0}\frac{(-1)^k}{k!}P_{1+k}(\lambda\vee \pi_\KKS^{\vee k})
	\end{align*}
and 
	\begin{align*}
	\pi_{\KKS,\red}
	:=
	\sum_{k\geq 0}\frac{(-1)^k}{k!}P_{k}(\pi_\KKS^{\vee k}).
	\end{align*}
There are now two issues with this approach:
	\begin{itemize}
	\item Since we did not introduce a complete filtration on the involved DGLAs,
		we have to check by hand that both of the series actually converge in a 
		suitable sense.
	\item This is actually not what we want, since our target, i.e. 
		$\Tpoly(M_\red)$, has to have zero curvature and zero differential. 
	\end{itemize}
This two problems are solved in Section~\ref{sec:Linfyquasiinverseofiota}, 
where we construct a quasi inverse of $\iota$ such that 
$\lambda_\red=\pi_{\KKS,\red} = 0$ and we show that 
the series are well-defined. But at first we need to extend 
our considerations to the formal setting, where we have 
a complete filtration by $\hbar$.

%
% subsection: formal setting
%

\subsection{Formal Equivariant Multivector Fields and Their Reduction}
\label{subsec:FormalmultivectorFieldsandReduction}

We want to consider the formal analogue of the 
equivariant equivariant multivector fields on $M$ from 
Eq.~\eqref{eq:curvedequitpolywithfixedJ}. Since we 
are only interested in formal Maurer-Cartan elements, 
we have to rescale the curvature by
$\hbar$, i.e. we consider the curved DGLA
\begin{equation*}
  ( (\Sym\liealg{g}^* \otimes \Tpoly(M))^\group{G}[[\hbar]], 
	\hbar\lambda , -[J,\argument],
	[\argument,\argument]).
\end{equation*}
A formal curved Maurer-Cartan elements $\hbar(\pi - J')
\in \hbar (\Sym\liealg{g}^* \otimes \Tpoly(M))^\group{G}[[\hbar]]$ corresponds 
to an invariant formal Poisson structure $\pi$ with formal momentum map 
$J + \hbar J'$. 

The Taylor series expansion discussed in
Section~\ref{sec:TaylorSeriesofTpoly1} allows us to interpret the
element $\hbar\pi_\KKS$ as a formal curved Maurer-Cartan
element. Thus we can perform the twisting procedure, yielding the
following flat DGLA:
\begin{align*}
  \left(T_\Tay(C\times\liealg{g}^*)[[\hbar]], [\hbar\pi_\KKS  - J, \argument], 
	[\argument,\argument]\right).
\end{align*} 
For a formal Maurer-Cartan element $\hbar(\pi - J')$ one can check
that $\pi_\KKS  + \pi$ is a $G$-invariant formal Poisson structure
with formal momentum map $J + \hbar J'$ as desired and again $\pi =
\pi_C + \mathcal{O}(\hbar)$.  Moreover, the Cartan model for the
multivector fields reads in the formal setting:
\begin{equation*}
  \left(T_\Cart(C)[[\hbar]], \hbar \del, 
	[\argument,\argument]\right)
\end{equation*}
and the bracket on $\Tpoly(M_\red)[[\hbar]]$ is simply extended
$\hbar$-bilinearly. Summarizing, we have the following claim.
\begin{theorem}
\label{thm:Tpolyredformal}
	We have built the following diagram
	\begin{equation*}
	\begin{tikzcd}
		  \left(T_\Tay(C\times\liealg{g}^*)[[\hbar]], 
			[\hbar\pi_\KKS  - J, \argument], 	[\argument,\argument]\right) \\
		  \left(T_\Cart(C)[[\hbar]], \hbar \del, 
	    [\argument,\argument]\right)
			\arrow[d,swap, "p"] 
			\arrow{u}{\iota}  \\
 	    (\Tpoly(M_\red)[[\hbar]],0,[\argument,\argument]), 
	\end{tikzcd}
\end{equation*}
where both maps are DGLA morphisms and where $\iota$ is still 
a quasi-isomorphism of DGLAs.
\end{theorem}
\begin{proof}
	The proof essentially follows from the above considerations.
        More explicitely, the inclusion of the Cartan model into
        $T_\Tay(C\times \liealg{g}^*)[[\hbar]]$ is a quasi-isomorphism
        of DGLAs since the bracket with $[-J,\argument]$ is not scaled
        by $\hbar$ and $[\hbar \pi_\KKS,\argument]$ is just
        $\hbar\del$ in the cohomology of $[-J,\argument]$. In other
        words, the argument from
        Proposition~\ref{prop:InclusionCtoMQuis} applies.
\end{proof}
Note that here we only use the fact that the $L_\infty$-quasi-inverse
of $\iota$ exists. In Section~\ref{sec:Linfyquasiinverseofiota} we
give an explicit formula for this map.
\begin{remark}[Laurent series]
  We observe that the map $p$ in the above theorem is \emph{not} a
  quasi-isomorphism due to the scaling problem by $\hbar$. Concerning
  the projection from the Cartan model to $M_\red$ we still have the
  map $h$ satisfying $ \hbar\del h + h \hbar \del =
  \hbar(\deg_\liealg{g} + \deg_\ver) \id, $ as in
  Proposition~\ref{prop:redfromctomred}. However, since we are not
  allowed to divide by $\hbar$, the projection $\pr$ in the formal
  setting is no longer a quasi-isomorphism. We remark that, if we
  consider instead Laurent series in $\hbar$ in all the complexes,
  e.g. $T_\poly(M_\red)[\hbar^{-1},\hbar]]$, then it remains a
    quasi-isomorphism.
\end{remark}
Moreover, we know from \cite[Thm.~4.6]{kontsevich:2003a} that 
$L_\infty$-quasi-isomorphisms 
induce bijections on the equivalence classes of formal Maurer-Cartan 
elements. In our setting this yields:
\begin{corollary}
  \label{cor:EquivtoTrivialMC}
  Every formal Maurer-Cartan element $\hbar(\pi - J')$ in
  $T_\Tay(C\times \liealg{g}^*)[[\hbar]]$ is equivalent to a formal
  Maurer-Cartan element $\hbar\pi_C\in T_\Cart(C)^1[[\hbar]] 
	\subset T_\Tay^1(C\times\liealg{g}^*)[[\hbar]]$.
\end{corollary}
In other words, the above Corollary states that every formal Poisson
structure $\pi_\KKS  + \pi$ with formal momentum map $J+ \hbar J'$
is equivalent to a formal Poisson structure $\pi_\KKS  + \pi_C$ with
undeformed momentum map $J$.
Finally, we can construct an explicit equivalence transformation from
a generic Maurer-Cartan element $\hbar(\pi - J')$ to one with $J'= 0$.
Set $X^1_\hbar = \hbar J'_i e^i$ and $J'^2_i = \exp(X^1_\hbar)(J_i) -
J_i - \hbar J'_i$.  One can recursively define for $k \geq 1$
\begin{equation}
  X_\hbar^{k+1} 
  =
  - J'^{k+1}_i e^i
  \coloneqq
  -\left( \exp(X_\hbar^k)\cdots \exp(X^1_\hbar) (J_i) - J_i - \hbar J'_i \right) 
  e^i.
\end{equation}
\begin{proposition}
  \label{prop:ExplicitEquivalenceXinfty}
  Let $\hbar(\pi - J')$ be a formal Maurer-Cartan element in 
	$T_\Tay(C\times\liealg{g}^*)[[\hbar]]$. 
	Then 
	\begin{equation}
	  X^\infty_\hbar
		=
		\log \left(\lim_{k \rightarrow \infty} 
		 \exp(X^k_\hbar) \cdots \exp(X^1_\hbar) \right)
	\end{equation} 
	satisfies $\exp(X^\infty_\hbar)(J_i) = J_i +\hbar J'_i$ and hence 
	$\hbar\exp(-X^\infty_\hbar)(\pi_\KKS  +\pi) - \hbar\pi_\KKS $ is a formal 
	Maurer-Cartan 
	element in $T_\Tay(C\times\liealg{g}^*)[[\hbar]]$ equivalent to $\hbar(\pi- J')$.
\end{proposition}
\begin{proof}
  Note that $X^1_\hbar \in \mathcal{O}(\hbar)$ and inductively one
  gets
  \begin{align*}
	  J_i + \hbar J_i'  + J_i'^{k+1}
		& =
    \exp(X^{k}_\hbar)\exp(X^{k-1}_\hbar)\cdots\exp(X^1_\hbar)(J_i) 
    =
    \exp(X^{k}_\hbar)(J_i +\hbar J'_i + J'^{k}_i) \\
		& = 
		J_i + \hbar J'_i + J_i'^k + X^k_\hbar(J_i) + \mathcal{O}(\hbar^{k+1}).
  \end{align*}
  Hence $J'^{k+1}_i  \in \mathcal{O}(\hbar^{k+1})$ as well as 
  $X^{k+1}_\hbar \in \mathcal{O}(\hbar^{k+1})$. In particular, $X^\infty_\hbar$ is 
  well-defined and satisfies 
  \begin{align*}
    \exp(X_\hbar^\infty) (J_i)
    =
    J_i +\hbar J'_i + \lim_{k \rightarrow \infty} J'^k_i
    =
    J_i + \hbar J'_i
  \end{align*}
  in the $\hbar$-adic topology. The gauge equivalence $\exp(-X_\hbar^\infty)$ therefore 
	maps $\hbar(\pi-J')$ to 
	\begin{align*}
	  \exp(-X_\hbar^\infty) \acts \hbar(\pi-J')
		& =
		\exp(-X_\hbar^\infty)(\hbar\pi_\KKS  -J + \hbar(\pi - J')) 
		- (\hbar\pi_\KKS -J) \\
		& =
		\hbar\exp(-X^\infty_\hbar)(\pi_\KKS  +\pi) - \hbar\pi_\KKS ,
	\end{align*}
	compare \cite[Prop.~6.2.34]{waldmann:2007a} for a formula of the gauge action.
\end{proof}

\subsection{$L_\infty$-quasi-inverse of $\iota$}
\label{sec:Linfyquasiinverseofiota}

Finally, we want to find an explicit description of the
$L_\infty$-quasi-inverse of $\iota$, i.e. an
$L_\infty$-quasi-isomorphism 
\begin{align*}
  P
  \colon
  (T_\Tay(C\times\liealg{g}^*),[\pi_\KKS -J,\argument],[\argument,\argument])
  \longrightarrow
  (\Tpoly(M_\red),0,[\argument,\argument]).
\end{align*}
One can check that the homotopy $h$ of
$[-J,\argument]$ from Proposition~\ref{prop:cohomologytrivialbracketj}
does not commute with $[\pi_\KKS, \argument]$. The idea is to start 
with $[-J,\argument]$ as differential on the
Taylor decomposition and zero differential on the Cartan model,
construct the $L_\infty$-quasi-isomorphism $P$ in this case, and then
investigate the compatibility with $[\pi_\KKS,\argument]$.

Let us focus on the following deformation retract of DGLA's
\begin{equation}
  \begin{tikzcd} 
	(T_\Cart(C),0)
	\arrow[rr, shift left, "i"] 
  &&   
  (T_\Tay(C\times\liealg{g}^*),[-J,\argument])
	\arrow[ll, shift left, "p"]
	\arrow[loop right, "h"] 
\end{tikzcd} 
\end{equation}
and apply the construction from Section~\ref{sec:ExplicitFormulasDefRetracts}.
By Proposition~\ref{prop:Infinityprojection} we have an
$L_\infty$-quasi-isomorphism $P$ given by $P_1 = p$ and
\begin{equation}
  \label{eq:Pinftygeneral}
  P_n
	=
	P_n^1 
	=
	(R_2^1P_n^2 -  P_{n-1}^1 Q_n^{n-1}) \circ H_n,
\end{equation}
where $Q$ and $R$ denote the $L_\infty$-structure on
$\Sym(T_\Tay(C\times\liealg{g}^*)[1])$ and on $\Sym(T_\Cart(C)[1])$, 
respectively. Moreover, $H_n$ is the extension of
\begin{align*}
  \begin{split}
	  h ( \xi \otimes P \otimes \alpha \otimes X)
	  =
	  \begin{cases}
	    \frac{-1}{\deg_{S\liealg{g}^*}\xi + \deg_{\Anti\liealg{g}^*}\alpha}
	    \inss(e_\ell)\xi \otimes P \otimes e^\ell	\wedge \alpha \otimes X 
	    \quad \; & 
	    \text{if} \deg_{S\liealg{g}^*}\xi + \deg_{\Anti\liealg{g}^*}\alpha \neq 0 \\
	    0 \quad \quad \quad \; &\text{else}
	  \end{cases}
	\end{split}
\end{align*}  
since $Q_1^1 = [J,\argument]$, compare Proposition~\ref{prop:cohomologytrivialbracketj}.
\begin{lemma}
  For $n=2$ one has
	\begin{equation}
	  P_2 (X_1 \vee X_2)
		=
		- p((-1)^{\abs{X_1}}[hX_1,X_2 ] - [X_1 ,hX_2])
	\end{equation}
	for all homogeneous $X_1,X_2 \in T_\Tay(C\times\liealg{g}^*)[1]$.
\end{lemma}
\begin{proof}
  One has $P_2^2 \circ H_2 = 0$. Furthermore, for
  $Q_2^1(X_1,Y_1) = - (-1)^{\abs{X_1}}[X_1,X_2]$ with $\abs{X_1}$
  denoting the shifted degree in $T_\Tay(C\times\liealg{g}^*)[1]$ we have 
	with the formula for $H_2$, see \cite[p.~383]{loday.vallette:2012a},
  \begin{align*}
    P_2 (X_1 \vee X_2)
    & = 
    -p \circ Q_2^1 \circ H_2 (X_1\vee X_2)  \\
    & =
    -\frac{p}{2}(-(-1)^{\abs{X_1}+1}[hX_1,X_2 + ip X_2] +(-1)^{\abs{X_1}+ \abs{X_1}+1} 
    [X_1 + ip X_1,hX_2])  \\
    & =
    - p((-1)^{\abs{X_1}}[hX_1,X_2 ] - [X_1 ,hX_2]).
  \end{align*}
  The last step is easily seen for homogeneous elements by counting the
  $\liealg{g}^*$-degrees. In fact, if $X_2 = ip X_2$, then $hX_2 = 0$
  and the statement holds.  If $ipX_2 = 0$, then $p([hX_1,X_2]) = 0$
  since the bracket contains at least one $\liealg{g}^*$-component
  that is annihilated by $p$.  The same holds for $1 \leftrightarrow
  2$.
\end{proof}

As a next step we want to obtain an $L_\infty$-morphism between
$(T_\Tay(C\times\liealg{g}^*),[\pi_\KKS-J,\argument])$ and
$(T_\Cart(C),\del)$.
Let us first observe that $P_n$ contains $n-1$ brackets and $n-1$
applications of $h$, increasing the $\Anti^\bullet
\liealg{g}^*$-degrees. This implies that the $P_n$ are non-zero only
if all $n$ arguments have no $\Anti\liealg{g}^*$-contribution and the
sum of the $\Sym\liealg{g}^*$-degrees is $n-1$. As a consequence, all
$n-1$ brackets consist of pairings between
$\Anti\liealg{g}^*$-components coming from $h$ and the $\prod \Sym
\liealg{g}$-components, whereas the $\Tpoly(C)$-components are just
wedged together.  Moreover, the first term in \eqref{eq:Pinftygeneral}
does not contribute since the bracket $R_2^1$ is here in $C$-direction
and we have
\begin{equation}
  \label{eq:Pinftysimpler}
  P_n
	=
	P_n^1 
	=
	-  P_{n-1}^1 \circ Q_n^{n-1}\circ H_n.
\end{equation}
Therefore, to prove the compatibility of $P$ with the differentials 
$[\pi_\KKS , \argument]$ and $\del$ we only have to show
\begin{equation*}
  -\del P_n^1
	=
	P_n^1 \circ (Q^\pi)_n^n,
\end{equation*}
where $(Q^\pi)_n^n$ is the extension of $-[\pi_\KKS,\argument]$. By the proof 
of Proposition~\ref{prop:bracketpiminusr} and the above arguments, the 
only part with a non-trivial contribution is the extension of
$-\del = -\id \otimes \inss(e^i) \otimes \id \otimes (e_i)_C \wedge$. 

\begin{proposition}
  \label{prop:PinftyquiswithPiKKS}
  The map $P$ from \eqref{eq:Pinftysimpler} is an
  $L_\infty$-quasi-isomorphism from the Taylor series expansion
  $(T_\Tay(C\times\liealg{g}^*),[\pi_\KKS-J,\argument])$ to
  $(T_\Cart(C),\del)$ and an $L_\infty$-quasi-inverse to the inclusion
  $\iota$ from Proposition~\ref{prop:InclusionCtoMQuis}. The same
  holds in the formal setting with the rescaled differentials
  $[\hbar\pi_\KKS-J,\argument]$ and $\hbar \del$.
\end{proposition}
\begin{proof}
  By the above reasoning all brackets consist of pairings in
  $\liealg{g}^*$-direction and the $\Tpoly(C)$-components are just
  wedged together, so $\del$ satisfies a Leibniz rule.  
  Let us show the statement inductively.	
  For $n=1$ it is obvious.  In addition, we know $[h,\del] = 0$ and thus $H_n
  (Q^\pi)_n^n = - (Q^\pi)_n^n H_n$.  If we prove $(Q^\pi)_n^n Q_{n+1}^n = -
  Q_{n+1}^n (Q^\pi)_{n+1}^{n+1}$ then \eqref{eq:Pinftysimpler} gives
  inductively
  \begin{equation*}
    -\del P_{n+1}^1
    =
    \del P_{n}^1 Q_{n+1}^{n} H_{n+1}
    =
    - P_{n} (Q^\pi)_n^n Q_{n+1}^n H_{n+1}
    =
    -P_n Q_{n+1}^n H_{n+1} (Q^\pi)_{n+1}^{n+1}.
  \end{equation*}
  We only have to show the desired $(Q^\pi)_n^n Q_{n+1}^n = - Q_{n-1}^n
  (Q^\pi)_n^n$ on the image of $H_{n+1}$ on elements with
  $\Anti\liealg{g}^*$-degree zero and with sum of
  $\Sym\liealg{g}^*$-degrees $n$, where $n \geq 1$. In particular,
  elements in this image have $\Anti\liealg{g}^*$-degree $1$ and
  $\Sym\liealg{g}^*$-degree $n-1$. Consider $X_1\vee\cdots \vee
  X_{n+1}$ where w.l.o.g. $X_1$ has a
  $\Anti\liealg{g}^*$-contribution, then the bracket has to be with
  respect to this vector field, the other terms vanish later under
  $p$.  Using \eqref{eq:Qniformula} we get as only non-vanishing
  contribution
  \begin{align*}
    Q_{n+1}^n(X_1 \cdots X_{n+1})
    =
    \sum_{i=2}^{n+1}
    (-1)^{\abs{X_i} (\abs{X_2}+ \cdots + \abs{X_{i-1}})}
    Q_{2}^1(X_1\vee X_i) \vee X_2 \vee \cdots \stackrel{i}{\vee} 
    \cdots \vee X_{n+1},
  \end{align*}
  where $\abs{X_i}$ denotes the degree in
  $T_\Tay(C\times\liealg{g}^*)[1]$.  A straightforward computation
  shows again
  \begin{equation*}
    -\del Q_2^1(X_1 \vee X_i)
    =
    Q_2^1(\del X_1 \vee X_i + (-1)^{\abs{X_1}} X_1 \vee \del X_i)
  \end{equation*}
  and combining these two expressions the desired result follows by a
  comparison of the signs of all terms involving $\del X_j$.
\end{proof}

\begin{remark}
  Note that here we can not use the usual twisting procedure since we have no
  complete filtration compatible to $P$ such that $\pi_\KKS$
  is of degree one. Of course, this is to be expected since the
  differential on $T_\Cart(C)$ is not an inner one.
\end{remark}
 
We can also show that $P$ is compatible with the curvature, which is 
easier to show in the formal setting.

\begin{proposition}
  \label{prop:PinftyFormalCurved}
  The map $P$ from \eqref{eq:Pinftysimpler} is an
  $L_\infty$-morphism between the curved DGLAs
  $(T_\Tay(C\times\liealg{g}^*)[[\hbar]],\hbar\lambda,[-J,\argument],[\argument,\argument])$
  and $(T_\Cart(C)[[\hbar]],0,\hbar\del,[\argument,\argument])$.
\end{proposition}
\begin{proof}
  We can twist $P$ from
  Proposition~\ref{prop:PinftyquiswithPiKKS} with $-\hbar\pi_\KKS$
  as in \cite[Lemma~2.7]{esposito.dekleijn:2018a:pre}.  Then we obtain
  an $L_\infty$-morphism
  $(T_\Tay(C\times\liealg{g}^*)[[\hbar]],\hbar\lambda,[-J,\argument])$
  to $(T_\Cart(C),0,\hbar\del)$. This is clear since the new
  codifferential on $\Sym (T_\Tay(C\times\liealg{g}^*)[[\hbar]][1])$ is
  given by
  \begin{align*}
    Q_0'
    & =
    Q_1(-\hbar\pi_\KKS) +\frac{1}{2} Q_2(-\hbar\pi_\KKS,-\hbar\pi_\KKS) \\
    & =
    [-\hbar\pi_\KKS-J,\hbar\pi_\KKS]
    =
    - \hbar \lambda \\
    Q_1' (X)
    & =
    Q_1(X) + Q_2(-\hbar\pi_\KKS,X)
    =
    [-\hbar\pi_\KKS +J +\hbar\pi_\KKS, X].
  \end{align*}
  Since $\pi_\KKS-R$ contains a $\Anti\liealg{g}^*$-degree the
  twisting does not change the $L_\infty$-structure on the Cartan
  model and the twisted morphism is just given by $P$.
\end{proof}
Note that in this case $P$ is no longer a quasi-isomorphism,
and that the result also holds in the classical setting:

\begin{corollary}
  \label{cor:PcompatiblewithcurvatureClassic}
  The map $P$ from \eqref{eq:Pinftysimpler} is also an
  $L_\infty$-morphism between the curved DGLAs
  $(T_\Tay(C\times\liealg{g}^*),\lambda,[-J,\argument])$ and
  $(T_\Cart(C),0,\del)$.
\end{corollary}
\begin{proof}
  Since the morphism $P$ is $\hbar$-linear we can compute
  explicitly that the Taylor coefficients of $P$ are compatible
  with the above curved DGLA structures. By the
  construction of $P$ we know
  \begin{align*}
    R_2^1 P_n^2 
    =
    P_n^1 Q_n^n + P_{n-1}^{1} Q_n^{n-1},
  \end{align*}
  where $R_2^1$ is the bracket on the Cartan model and $Q_1^1$ is the
  extension of $[J,\argument]$. Moreover, we have by
  Proposition~\ref{prop:PinftyFormalCurved}
  \begin{align*}
    \hbar R_1^1 P_n^1 + R_2^1 P_n^2 
    =
    P_{n+1}^1(\hbar Q_0 \vee \argument) +
    P_n^1 Q_n^n + P_{n-1}^{1} Q_n^{n-1},
  \end{align*}
  where $R_1^1 = - \del$ and $Q_0 = -\lambda$. This gives 
  \begin{align*}
    \hbar R_1^1 P_n^1
    =
    P_{n+1}^1(\hbar Q_0 \vee \argument)
    \quad \Longrightarrow \quad
    R_1^1 P_n^1
    =
    P_{n+1}^1( Q_0 \vee \argument)
  \end{align*}
  and the statement is shown.
\end{proof}

\begin{remark}\label{Rem: Pcurved}
  This can also be directly shown for the classical setting. Indeed,
  we do not have the complete filtration, but by the explicit forms of
  $P$ and $\pi_\KKS $ all the appearing series in the
  twisting procedure are still well-defined.
\end{remark}

\section{The Reduction $L_\infty$-Morphism and Reduction of Formal Poisson Structures}
\label{sec:ReductionMorphism}

Let us now merge together all the results we obtained in the 
previous sections in order to finalize the construction of the reduction scheme. 
Given a Lie group action $\Phi\colon G\times M\to M$ 
on a general manifold $M$ and an equivariant map $J\colon M \to \liealg{g}^*$ with 
value and regular value $0$ interpreted as an element 
$J\in (\liealg{g}^*\tensor\Cinfty(M))^G$. 
In \eqref{eq:curvedequitpolywithfixedJ} we defined the curved differential 
graded Lie algebra
	\begin{align*}
	  (T_{\liealg{g}}(M),\lambda,-[J,\argument],[\argument,\argument]),
	\end{align*}	 
and we want to obtain an $L_\infty$-morphism to $T_\poly(M_\red)$ with 
zero differential in order to reduce in particular formal Poisson structures.

\subsection{The Reduction $L_\infty$-morphism}

Under the above assumptions 
that the action is proper in an open neighbourhood of 
the constraint surface $C:=J^{-1}(\{0\})$, we find an open 
$\group{G}$-invariant neighbourhood 
$C\subseteq M_\nice  \cong U_\nice  \subseteq C\times \liealg g^*$, 
such that the momentum map on $U_\nice$ is just
the projection on the second factor and such the group acts as 
the product of the action on $C$ and the coadjoint action. 
This yields the curved DGLA morphism 
\begin{align*}
	\cdot\at{U_\nice }\colon  
	(T_{\liealg{g}}(M),\lambda,-[J,\argument],[\argument,\argument])
	\longrightarrow
	 (T_{\liealg{g}}(U_\nice ),
	 \lambda\at{U_\nice },-[J\at{U_\nice },\argument],[\argument,\argument])
\end{align*}
which is just the restriction to the invariant open subset 
$M_\nice $ concatenated with the extension of the $\group{G}$-equivariant 
diffeomorphism to $U_\nice$. Moreover, we know from 
\cite[Lemma~3]{bordemann.herbig.waldmann:2000a} that $U_\nice$ is an 
open neighbourhood of $C\times \{0\}$ such that $U_\nice \cap (\{p\} \times 
\liealg{g}^*)$ is star-shaped around $\{p\} \times \{0\}$ for all 
$p \in C$, hence we also have the Taylor expansion 
as in Equation~\eqref{eq:TaylorSeries}. It is a morphism of 
curved DGLA's
	\begin{align*}
	   T_{\liealg{g}^*}\colon 
	   (T_{\liealg{g}}(U_\nice ),
	 	\lambda\at{U_\nice },-[J\at{U_\nice },\argument],
	 	[\argument,\argument]))
	 	\longrightarrow
	 	\big(T_\Tay (C\times \liealg{g}^*),\lambda,-[J,\argument],
	 	[\argument,\argument]\big).
\end{align*}
With Propostion~\ref{prop:PinftyFormalCurved} and 
Corollary~\ref{cor:PcompatiblewithcurvatureClassic}, 
we obtain furthermore a curved $L_\infty$-morphism 
	\begin{align*}
	P\colon 
	\big(T_\Tay (C\times \liealg{g}^*),\lambda,-[J,\argument],
	[\argument,\argument]\big)
	\longrightarrow
	(T_\Cart(C),0,\del,[\argument,\argument])
	\end{align*}
and finally we have the projection $p\colon(T_\Cart(C),0,\del,[\argument,\argument])\to 
(\Tpoly(M_\mathrm{red}),0,0,[\argument,\argument])$ from 
Equation~\eqref{Eq: DefRetrCart} that is a DGLA 
morphism and hence also a morphism of (curved) $L_\infty$-algebras.

\begin{theorem}
  \label{thm:ClassicalTred2}
  The concatenation of all the above morphism results in a curved $L_\infty$-morphism 
	\begin{align}
	  \mathrm{T}_\red\colon 
		(T_{\liealg{g}}(M),\lambda,-[J,\argument],  [\argument,\argument])
	  \longrightarrow 
	  (\Tpoly(M_\mathrm{red}),0,0,[\argument,\argument]),
	\end{align}	   
	called \emph{reduction $L_\infty$-morphism}. Considering the setting of 
	formal power series in $\hbar$ we can extend $\mathrm{T}_\red$ $\hbar$-linearly and obtain  
	\begin{align*}
		\mathrm{T}_\red\colon (T_{\liealg{g}}(M)[[\hbar]],\hbar\lambda,-[J,\argument],
		[\argument,\argument])
		\longrightarrow
		(\Tpoly(M_\mathrm{red})[[\hbar]],0,0,
		[\argument,\argument]).
	\end{align*}
\end{theorem}

\subsection{Reduction of Formal Poisson Structures}

As mentioned above, a formal curved Maurer-Cartan element 
$\hbar(\pi-J')\in \hbar T_{\liealg{g}}(M)[[\hbar]]$ is an 
invariant formal Poisson structure $\hbar\pi$ with formal moment map $J+\hbar J'$. 
By $\mathrm{T}_\red$ we obtain therefore a formal Maurer-Cartan element 
\begin{align}
	\hbar\pi_\mathrm{red}
	=
	\sum_{k\geq 1}\frac{1}{k!} \mathrm{T}_{\red,k}(\hbar(\pi-J')^{\vee k})
\end{align}	 
in $\Tpoly(M_\mathrm{red})[[\hbar]]$ which corresponds to a formal Poisson structure 
$\pi_\red$ on $M_\red$. 

In order to show that this morphism gives indeed a non-trivial 
reduction scheme for formal Poisson structures we show at first that we 
recover the Marsden-Weinstein reduction. This classical setting is included in our formulation 
by considering special curved formal Maurer-Cartan elements 
$\hbar \pi \in \hbar T_\liealg{g}(M)[[\hbar]]$, 
where in fact $\pi \in T^1_\poly(M)$ does not depend on $\hbar$, i.e. is 
a classic $\group{G}$-invariant Poisson structure with momentum map $\pi$. 

\begin{proposition}
  \label{prop:ClassicalComparison}
  The reduction procedure of Marsden-Weinstein coincides with the one via $\mathrm{T}_\red$ from 
  Theorem~\ref{thm:ClassicalTred2} for Maurer-Cartan elements of the 
	form $\hbar \pi \in \hbar T_\liealg{g}(M)[[\hbar]]$ with $\pi \in T^1_\poly(M)$.
\end{proposition}
\begin{proof}
  By Lemma~\ref{lemma:TaylorexpansionPoissonnormalform} we know that 
	$\hbar\pi$ takes in the Taylor expansion $\hbar\pi$ the form 
	$\hbar \pi_\KKS + \hbar \pi_C$, where 
  $\pi_C = \prod_i \pi_C^i$
  with $\pi_C^i \in \Sym^i \liealg{g} \otimes T_\poly^1(C)$. Then 
	the application of $p \circ P$ yields a Maurer-Cartan element
  $\hbar\pi_\red$ in the reduced DGLA
  $(T_\poly(M_\red[[\hbar]]),0,[\argument,\argument])$ via
  \begin{equation*}
    \hbar\pi_\red
    =
    \sum_{k\geq 1} \frac{1}{k!} p \circ P_k(\hbar(\pi_\KKS+\pi_C),\dots,\hbar(\pi_\KKS+\pi_C))
    =
    p( \hbar\pi_C^0),
  \end{equation*}
  so this series is indeed well-defined. This Maurer-Cartan element corresponds 
	to a classical Poisson structure $\pi_\red$ with
  \begin{align*}
    p^*\pi_\red(\D \phi, \D \psi)
    & =
    \pi_C^0 (\D p^*\phi, \D p^*\psi)
    =
    \iota^*((\pi_\KKS+ \pi_C)(\D \prol p^*\phi, \D\prol p^*\psi)
  \end{align*}
  for $\phi,\psi \in \Cinfty(M_\red)$, where $\prol \colon \Cinfty(C) 
	\rightarrow \Cinfty(C) \otimes \prod_i \Sym^i \liealg{g}$ is the 
	canonical prolongation. But this
  is just the usual reduced Poisson structure 
	from Marsden-Weinstein reduction.
\end{proof}

Now we want to show that our construction is 
indeed a non-trivial extension of the classical Marsden-Weinstein 
reduction to the formal setting. For simplicity, let us  
consider for a moment just a part of $\mathrm{T}_\red$, namely the map
\begin{align*}
	\tilde{\mathrm{T}}_\red = p \circ P \colon 
	\big(T_\Tay (C\times \liealg{g}^*)[[\hbar]],\hbar\lambda,-[J,\argument],
	[\argument,\argument]\big)
	\longrightarrow
	(\Tpoly(M_\mathrm{red})[[\hbar]],0,0,
	[\argument,\argument]).
\end{align*}

\begin{lemma}
	The induced map at the level of Maurer--Cartan elements 
	\begin{align}
	\begin{split}
	  \tilde{\mathrm{T}}_\red
	  \colon
	  MC(T_\Tay (C\times \liealg{g}^*)[[\hbar]])
	  & \longrightarrow
	  MC(\Tpoly(M_\mathrm{red})[[\hbar]])  \\
		\hbar (\pi - J') 
	  &\longmapsto
	  \sum_{k\geq 1}\frac{1}{k!}\tilde{\mathrm{T}}_{\red,k}((\hbar (\pi -J')^{\vee k})
	\end{split}
	\end{align}
	is a surjection. 
	\end{lemma}
\begin{proof}
Let $\hbar\pi_\red\in MC(\Tpoly(M_\mathrm{red})[[\hbar]])$, then we 
know from Corollary~\ref{Cor: SurjProj} that
\begin{align*}
	\hbar \Pi
	=
	\sum_{k\geq 1}\frac{1}{k!} \iota_{\infty,k}((\hbar\pi_\red)^{\vee k})
\end{align*} 
is a well-defined Maurer--Cartan element in $T_\Cart(C)[[\hbar]]$
 with $p(\hbar\Pi)=\hbar\pi$. Using 
Proposition~\ref{prop:bracketpiminusr} we see that 
$\hbar(\pi_\KKS +\Pi) \in MC(T_\Tay(C\times\liealg{g})[[\hbar]])$ 
and
\begin{align*}
  \sum_{k\geq 1}\frac{1}{k!}\tilde{\mathrm{T}}_{\red,k}((\hbar(\pi_\KKS+ \Pi))^{\vee k})
	=
	p(\hbar\Pi)
	=
	\pi
\end{align*}
as desired.
\end{proof}
\subsection{Comparison of the Reduction Procedures}
\label{sec:ComparisonofRedProd}

We conclude with a comparison of the different reduction procedures. 
More explicitly, we want to compare the reduction via $\mathrm{T}_\red$ 
from Theorem~\ref{thm:ClassicalTred2} with the 
with the reduction of formal Poisson structures via the homological
perturbation lemma, see Appendix~\ref{sec:BRSTlikeReduction}.

In the setting of curved DGLA's or curved $L_\infty$-algebras it is 
more tricky to talk about equivalent Maurer-Cartan elements. Thus
we switch to the description of our reduction in terms of 
flat DGLA's as in Theorem~\ref{thm:Tpolyredformal}. Here we need 
$\pi_\KKS$ which is not available in the general setting, so from now 
on we restrict ourselfes to the Taylor expansion $\left(T_\Tay(C\times\liealg{g}^*),
[\hbar\pi_\KKS-J,\argument], [\argument,\argument]\right)$.

Consider formal Poisson structure $\pi_\hbar = \sum_{r=0}^\infty
\hbar^r \pi_r \in \Secinfty(\Anti^2 TM)[[\hbar]]$ with formal
equivariant momentum map $J_\hbar = J + \hbar J'
\colon \liealg{g} \rightarrow \Cinfty(M)[[\hbar]]$. 
By Proposition~\ref{prop:reducedpoissonvianormalizer} 
one gets an induced formal Poisson
bracket on $M_\red = J^{-1}(\{0\})/\group{G}$ via
\begin{equation*}
  \pi^*\{u,v\}_\red
  =
  \boldsymbol{\iota^*}\{[\prol\pi^*u],[\prol\pi^*v]\}_\hbar,
\end{equation*}
where the deformed restriction map is given by
\begin{equation}
  \label{eq:DeformedRestriction}
  \boldsymbol{\iota^*}
  = 
  \iota^*(\id +\insa(\hbar J') h_0)^{-1}
	=
	\iota^* \sum_{k=0}^\infty (- \insa(\hbar J' )h_0)^k ,
\end{equation}
compare Proposition~\ref{prop:reducedpoissonvianormalizer}. We directly see 
that the reduction procedure works analogously for 
$\pi_\hbar \in T_\Tay^1(C\times \liealg{g}^*)[[\hbar]]$.

\begin{theorem}
  \label{thm:FormalReductionsEqual}
  The reduction of formal equivariant Poisson structures with formal momentum 
	maps via 
	\begin{align*}
	  \tilde{\mathrm{T}}_\red = p \circ P \colon 
	  \big(T_\Tay (C\times \liealg{g}^*)[[\hbar]],[\hbar\pi_\KKS-J,\argument],
	  [\argument,\argument]\big)
	  \longrightarrow
	  (\Tpoly(M_\mathrm{red})[[\hbar]],0,
	  [\argument,\argument])
  \end{align*}
	coincides with the reduction of formal Poisson structures via the homological 
	perturbation lemma from Proposition~\ref{prop:reducedpoissonvianormalizer}.
\end{theorem}
\begin{proof}
  We show at first that the reduction 
	procedures coincide on Maurer-Cartan elements 
	of the form $\hbar\pi_C$, i.e. where the quantum momentum map is just the 
	classical momentum map. Note that by Corollary~\ref{cor:EquivtoTrivialMC} every formal 
	Maurer-Cartan element $\hbar(\pi'-J')$
  is equivalent to such a $\hbar\pi_C$. Writing again $\pi_C^i \in 
	(\Sym^i\liealg{g}\otimes  T_\poly^1(C))[[\hbar]]$, the 
	reduced Poisson structure via $\tilde{\mathrm{T}}_\red$ is easy to describe, namely by
  \begin{equation*}
    \hbar\pi_\red
    =
		\sum_{k=1}^\infty \frac{1}{k!} \tilde{\mathrm{T}}^1_{\red,k}(\hbar\pi_C \vee \cdots \vee \hbar\pi_C)
		=
    \sum_{k=1}^\infty \frac{\hbar^k}{k!} p\circ P_k(\pi_C,\dots,\pi_C)
    =
    p( \hbar\pi_C^0).
  \end{equation*}
	In the reduction via the homological perturbation lemma one has  
  $\boldsymbol{\iota^*} = \iota^*$ and thus the reduced formal Poisson
  structures coincide by the same reasons as in the classical setting of 
	Proposition~\ref{prop:ClassicalComparison}. 
	
	The idea is now to use the explicit equivalence from 
	Proposition~\ref{prop:ExplicitEquivalenceXinfty}. Let $\hbar(\pi - J')$ be a 
	formal Maurer-Cartan element in $T_\Tay(C\times\liealg{g}^*)[[\hbar]]$ 
	and $X^\infty_\hbar$ be the equivalence between the formal Maurer-Cartan 
	elements $(\pi_\KKS  + \pi, J+\hbar J')$ and 
	$(\pi_\KKS + \pi_C, J)$. 
	The reduction via the homological perturbation lemma maps both 
	Poisson structures to the same formal Poisson structure on $M_\red$. This 
	follows from Formula~\eqref{eq:HomPerRedofEquivalences} for the equivalence 
	between the reduced Poisson structures since $X^\infty_\hbar$ differentiates only in 
	direction of $\liealg{g}^*$.
	We only have to show that $\tilde{\mathrm{T}}_\red$ also maps both to the same one.
  But $X_\hbar^\infty$ induces the following equivalence on the level 
	of the reduced manifold
	\begin{equation*}
	  p \circ P^1(X^\infty_\hbar \vee \exp(\exp(X^\infty_\hbar) \acts \hbar(\pi-J')))
		=
		0,
	\end{equation*}
	see e.g. \cite[Prop.~4.9]{canonaco:1999a}, whence both reduced structures are again 
	equal. This proves the theorem.
\end{proof}

\appendix 
\section{BRST-Like Reduction of Formal Poisson Structures}
\label{sec:BRSTlikeReduction}

In this section we want to recall a reduction scheme for formal
Poisson structures similarly to the reduction of star products in
\cite{gutt.waldmann:2010a} resp. to the BRST reduction as formulated
in \cite{bordemann.herbig.waldmann:2000a}. We recall at first the
homological perturbation lemma adapted to our setting, see
\cite[Thm.~2.4]{crainic:2004a:pre} and
\cite[Chapter~2.4]{reichert:2017b}.

\subsection{Homological Perturbation Lemma}

\begin{definition}[Homotopy equivalence data]
  \label{defi:homotopyequivalencedata}
  A \emph{homotopy equivalence data} (HE data) consists of two chain
  complexes $(C,\D_C)$ and $(D,\D_D)$ over a commutative ring
  $\ring{R}$ together with two quasi-isomorphisms
  \begin{equation}
    p \colon 
    C 
    \longrightarrow 
    D 
    \quad\text{and}\quad 
    i \colon 
    D 
    \longrightarrow 
    C
  \end{equation}
  and a chain homotopy 
  \begin{equation}
    h \colon 
    D 
    \longrightarrow 
    D 
    \quad \text{with} \quad
    \id_D - pi 
    = 
    \D_D h + h \D_D
  \end{equation}
  between $pi$ and $\id_D$.
\end{definition} 
For a shorter notation we will denote such a HE data by 
\begin{equation*}
  p \colon 
  (C,\D_C) 
  \rightleftarrows 
  (D,\D_D)  \colon i,h.
\end{equation*}
Moreover, we say that a graded map $B\colon D_\bullet \longrightarrow
D_{\bullet -1}$ with $(\D_D + B)^2=0$ is a \emph{perturbation} of the
HE data. The perturbation is called \emph{small} if $\id_D + B h$ is
invertible, and the homological perturbation lemma states that in this
case the perturbed HE data is a again a HE data, see
\cite[Thm.~2.4]{crainic:2004a:pre} for a proof.

\begin{proposition}[Homological perturbation lemma]
  \label{lemma:homologicalperturbationlemma}
  Let
  \begin{equation*}
    p \colon 
    (C,\D_C) 
    \rightleftarrows 
    (D,\D_D)  \colon i,h
  \end{equation*}
  be a HE data and let $B$ be small perturbation of $\D_D$, then the
  perturbed data
  \begin{equation}
    P \colon 
    (C,\widehat{\D}_C) 
    \rightleftarrows 
    (D,\widehat{\D}_D)  \colon I,H
  \end{equation}
  with
  \begin{equation}
    \begin{alignedat}{5}
      \label{eq:deformedhomologicalmaps}
      A 
      &= 
      (\id_D + Bh)^{-1}B, \quad
      &&
      \widehat{\D}_D 
      &&= 
      \D_D + B, \quad
      &&
      \widehat{\D}_C 
      &&= 
      \D_C + i Ap   ,    \\
      P
      &= 
      p -hAp ,
      &&\; 
      I 
      &&=  
      i-iAh  ,
      && 
      H 
      &&= 
      h-hAh  
    \end{alignedat}
  \end{equation}
  is again a HE data.
\end{proposition}

We will even encounter a simpler situation, namely that the complex
$C$ is concentrated in degree $0$ and $D_n=0$ for $n<0$:

\begin{equation}
  \label{eq:simplehedata}
  \begin{tikzcd}[row sep = large, column sep = large]
    0    	&
    D_0   \arrow[l]
    \arrow[d,shift left=-.75ex,swap,"i"] 
    \arrow[r,shift left=-.75ex,swap,"h_0"]      &		
    D_1   \arrow[l,shift left=-.75ex,swap,"\D_{D,1}"]
    \arrow[r,shift left=-.75ex,swap,"h_1"]  &    
    \cdots\arrow[l,shift left=-.75ex,swap,"\D_{D,2}"]  \\
    0     &		 
    C_0   \arrow[l] 
    \arrow[u,shift left=-.75ex,swap,"p"]  	&	
    0	    \arrow[l]	 	&          	
  \end{tikzcd}
\end{equation}
In this case, the perturbed HE data corresponding to a small
perturbation $B$ according to \eqref{eq:deformedhomologicalmaps} is
given by
\begin{equation*}
  P 
  = 
  p, \quad 
  I 
  = 
  i - i(\id_D + B_1 h_0)^{-1}B_1h_0 , \quad
  H 
  = 
  h-h(\id_D + Bh)^{-1}Bh
\end{equation*}
and, using the geometric power series, this can be simplified to
\begin{equation}
  \label{eq:simplifieddeformedhomologicalstuff}
  P 
  = 
  p, \quad 
  I 
  = 
  i(\id_D + B_1 h_0)^{-1}, \quad 
  H 
  = 
  h(\id_D + Bh)^{-1}.
\end{equation}	
Here we denote by $B_1\colon D_1 \longrightarrow D_0$ the degree one
component of $B$, analogously for $h$.

\subsection{Formal Koszul Complex}

We start with the classical Koszul complex $\Anti^\bullet\liealg{g}
\otimes \Cinfty(M)$ that can be interpreted as the smooth functions on
$M$ with values in the complexified Grassmann algebra of
$\liealg{g}$. The Koszul differential $\del$ is given by
\begin{equation}
  \del \colon 
  \Anti^q\liealg{g}\otimes \Cinfty(M) 
  \longrightarrow
  \Anti^{q-1}\liealg{g}\otimes \Cinfty(M) , \hspace{5pt} 
  a 
  \mapsto 
  \ins(J_0)a = J_{0,i} \ins(\basis{e}^i)a,
  \label{eq:koszuldifferential}
\end{equation}
where $\ins$ denotes the left insertion and $J_0 = J_{0,i}\basis{e^i}$
the decomposition of $J_0$ with respect to a basis
$\basis{e}^1,\dots,\basis{e}^n$ of $\liealg{g}^*$. The corresponding
dual basis will be denoted by $\basis{e}_1,\dots,\basis{e}_n$ and
$\del^2 =0$ follows immediately with the commutativity of the
pointwise product in $\Cinfty(M)$. The differential $\del$ is also a
derivation with respect to associative and super-commutative product
on the Koszul complex, consisting of the $\wedge$-product on
$\Anti^\bullet \liealg{g}$ tensored with the pointwise product on the
functions.  Moreover, it is invariant with respect to the induced
$\liealg{g}$-representation
\begin{equation}
	 \label{eq:ActionofgonKoszulComplex}
	 \liealg{g}\ni \xi 
	 \mapsto 
	 \rho(\xi)
	 =
	 \ad(\xi) \otimes \id - \id \otimes \Lie_{\xi_M} \in 
	 \End(\Anti^\bullet \liealg{g}\otimes \Cinfty(M))
\end{equation}
as we have
\begin{align*}
  \del \rho(\basis{e}_a) ( x \otimes f )
	& =
	f^k_{aj} \basis{e}_k\wedge \ins(\basis{e}^j)\wedge\ins(\basis{e}^i) x \otimes J_{0,i}f 
	+ f^i_{aj}\ins(\basis{e}^j)x \otimes J_{0,i} f
	+ \ins(\basis{e}^i)x \otimes J_{0,i} \{J_{0,a},f\}_0       \\
	& =
	\rho(\basis{e}_a)\del ( x \otimes f)
\end{align*}
for all $x \in \Anti^\bullet \liealg{g}$ and $f\in \Cinfty(M)$.

One can show that the Koszul complex is acyclic in positive degree
with homology $\Cinfty(C)$ in order zero, and that one has a
$\group{G}$-equivariant homotopy
\begin{equation}
  h_i\colon 
  \Anti^i  \liealg{g} \otimes \Cinfty(M) 
  \longrightarrow 
  \Anti^{i+1}  \liealg{g} \otimes \Cinfty(M) 
\end{equation}
given on $C \subset M_\nice \subset C \times \liealg{g}^*$ by
\begin{equation*}
  h_k(x)(c,\mu)
  = 
  \basis{e}_i \wedge \int_0^1 t^k 
  \frac{\del x}{\del \mu_i}(c,t\mu) \D t,
  \quad \text{with} \quad 
  \del h_0 = \id_0 - \prol \iota^*
  \quad \text{and} \quad 
  h_0 \circ \prol = 0,
\end{equation*}
where $x \in \Anti^k \liealg{g} \otimes \Cinfty(C\times\liealg{g}^*)$
and $(c,\mu) \in C \times\liealg{g}^*$, see
\cite[Lemma~6]{bordemann.herbig.waldmann:2000a} and
\cite{gutt.waldmann:2010a} for the notation $M_\nice$.  In other
words, this means that
\begin{equation*}
	\prol \colon 
	(\Cinfty(C),0) 
	\rightleftarrows 
	(\Anti^\bullet \liealg{g}\otimes \Cinfty(M),\del)  
	\colon \iota^*,h
\end{equation*}
is a HE data of the special type of \eqref{eq:simplehedata}, i.e. we
have the following diagram:
\begin{figure*}[h]
	\center
	\begin{tikzcd}[row sep = large, column sep = large]
			0    	&
			\Cinfty(M) \arrow[l]\arrow[d,shift left=-.75ex,swap,"\iota^*"] 
			\arrow[r,shift left=-.75ex,swap,"h_0"]      &		
			\Anti^1\liealg{g}\otimes\Cinfty(M) \arrow[l,shift left=-.75ex,swap,"\del_1"]
			\arrow[r,shift left=-.75ex,swap,"h_1"]  &    
			\cdots \arrow[l,shift left=-.75ex,swap,"\del_2"]  \\
			0     &		 
			\Cinfty(C)  \arrow[l] \arrow[u,shift left=-.75ex,swap,"\prol"]  	&	
			0 \arrow[l]	 	&          	
	\end{tikzcd}
\end{figure*}

Let now $\pi_\hbar$ be an invariant formal Poisson structure with
formal equivariant momentum map $J_\hbar$. In order to take care of
the formal momentum map, we extend the Koszul complex $\hbar$-linearly
and gain the HE data
\begin{equation*}
  \prol \colon (\Cinfty(C)[[\hbar]],0) \rightleftarrows 
  (\Anti^\bullet \liealg{g}\otimes \Cinfty(M)[[\hbar]],\del) 
  \colon \iota^*,h.
\end{equation*}
Since the formal momentum map $J_\hbar$ is a deformation of $J_0$ in
the sense that the difference $J_\hbar - J_0 = J' \colon \liealg{g}
\rightarrow \hbar \Cinfty(M)[[\hbar]]$ starts in order one of $\hbar$,
the formal differential $\del_\hbar = \ins(J_\hbar) = \del + B$ with
$B= \ins(J')$ on $\Anti^\bullet \liealg{g} \otimes
\Cinfty(M)[[\hbar]]$ is a small perturbation in the sense of the
homological perturbation lemma
\ref{lemma:homologicalperturbationlemma}. Indeed, $\del_\hbar^2 =0$
follows for the same reasons as $\del^2=0$, and $\id + Bh$ is
invertible as formal power series since $Bh$ stars in order one of
$\hbar$. Consequently, the corresponding perturbed HE data
\begin{equation*}
  \boldsymbol{\prol} \colon 
	(\Cinfty(C)[[\lambda]],0) 
	\rightleftarrows 
  (\Anti^\bullet \liealg{g}\otimes \Cinfty(M)[[\lambda]],\del_\hbar)
  \colon \boldsymbol{\iota^*},\boldsymbol{h}
\end{equation*}
is given by
\begin{equation}
  \label{eq:deformedmapsbyhompertlemma}
  \boldsymbol{\prol} 
	= 
	\prol, \quad 
  \boldsymbol{\iota^*}
	= 
	\iota^*(\id +B_1 h_0)^{-1} ,  \quad
  \boldsymbol{h} 
	= 
	h(\id +B h)^{-1},
\end{equation}	
compare \eqref{eq:simplifieddeformedhomologicalstuff}. In particular, 
we have $\boldsymbol{\iota^*} \del_\hbar = 0$, 
\begin{equation}
  \label{eq:defChainHom}
  \id_{\Anti^\bullet \liealg{g}\otimes\Cinfty(M)[[\hbar]]}
	- \prol \boldsymbol{\iota^*} 
	= 
	\del_\hbar \boldsymbol{h} + \boldsymbol{h} \del_\hbar
\end{equation}
as well as $\boldsymbol{\iota^*}\prol= \id_{\Cinfty(C=[[\hbar]])}$
because of $h_0 \prol = 0$. Moreover, $\del_\hbar$ is still a
$\liealg{g}$-equivariant derivation of the algebra
structure. Therefore, also $\boldsymbol{\iota^*}$ and $\boldsymbol{h}$
are $\liealg{g}$-equivariant as all involved maps are.

We denote the image of the deformed Koszul differential by
\begin{equation*}
  \mathcal{J}_\hbar 
  =
  \image \del_\hbar\at{\Anti^1\liealg{g} \otimes \Cinfty(M)[[\hbar]]} 
  =
  \SP{J_i}_i.
\end{equation*}
Since $\prol \boldsymbol{\iota^*}$ is a projection with kernel
$\mathcal{J}_\hbar$, compare \eqref{eq:defChainHom}, we get with the
injectivity of $\prol$
\begin{equation*}
  \mathcal{J}_\hbar
  =
  \ker \boldsymbol{\iota^*}\at{\Cinfty(M)[[\hbar]]}.
\end{equation*}
As $\del_\hbar$ is $\Cinfty(M)[[\hbar]]$-linear, $\mathcal{J}_\hbar$
is an ideal in $\Cinfty(M)[[\hbar]]$ with respect to the pointwise
product.  Moreover, $\mathcal{J}_\hbar$ is a Poisson subalgebra of
$(\Cinfty(M)[[\hbar]], \{\,\cdot\,, \,\cdot\,\}_\hbar)$ because of
\begin{align*}
  \boldsymbol{\iota^*} \{f,g\}_\hbar
  & =
  \boldsymbol{\iota^*}( f^i g^j \{J_i,J_j\}_\hbar + f^iJ_j\{J_i,g^j\}_\hbar 
  + J_ig^j\{f^i,J_j\}_\hbar + J_iJ_j\{f^i,g^j\}_\hbar)
  = 
  0
\end{align*}
for $f = f^i J_i, g = g^j J_j \in \mathcal{J}_\hbar$.  As usual, one
can consider the Poisson normalizer
\begin{equation*}
  \mathcal{B}_\hbar 
  =
  \{f\in \Cinfty(M)[[\hbar]] \mid 
  \{f, \mathcal{J}_\hbar\} \subset \mathcal{J}_\hbar \},
\end{equation*} 
the biggest Poisson subalgebra containing $\mathcal{J}_\hbar$ as
Poisson ideal.  Then we know that the quotient is a Poisson algebra
and we even have the following:

\begin{proposition}
  \label{prop:reducedpoissonvianormalizer}
  There exists a unique formal Poisson structure $\pi_\red$ on 
  $M_\red$ such that
  \begin{equation*}
    \mathcal{B}_\hbar / \mathcal{J}_\hbar 
    \ni [f]
    \longmapsto
    \boldsymbol{\iota^*}f \in 
    \pi^*\Cinfty(M_\red)[[\hbar]]
  \end{equation*}
  is an isomorphism of Poisson algebras with inverse $\pi^*u \mapsto
  [\prol \pi^* u]$.
\end{proposition}
\begin{proof}
  We have for $u \in \Cinfty(M_\red)[[\hbar]], j = j^k J_k \in
  \mathcal{J}_\hbar$ and $f \in \mathcal{B}_\hbar$
  \begin{equation*}
    \boldsymbol{\iota^*}\{\prol \pi^* u, j\}_\hbar
    =
    \boldsymbol{\iota^*}(j^k\{\prol \pi^* u, J_k\}_\hbar + 
    J_k\{ \prol\pi^* u , j^k\}_\hbar)
    =
    \boldsymbol{\iota^*}(j^k \Lie_{(e_k)_M}\prol\pi^*u)
    =
    0
  \end{equation*}
  as well as 
  \begin{equation*}
    \Lie_{(e_i)_C}\boldsymbol{\iota^*} f
    =
    \boldsymbol{\iota^*} \Lie_{(e_i)_M}f
    =
    \boldsymbol{\iota^*} \{f, J_i\}_\hbar 
    =
    0,
  \end{equation*}
  thus the maps are both well-defined.  The fact that the maps are
  mutually inverse is clear since
  \begin{equation*}
    \boldsymbol{\iota^*}\prol
    =
    \id
    \quad \text{and} \quad
    \id- \prol \boldsymbol{\iota^*} 
    = 
    \del_\hbar \boldsymbol{h} \in 
    \mathcal{J}_\hbar.
  \end{equation*}
  The compatibility with the pointwise product follows from the
  explicit form $\boldsymbol{\iota^*}= \iota^* \circ \sum_k (- B_1
  h_0)^k$ and the fact that
  \begin{equation*}
    h_0( f \prol \phi)
    =
    \prol\phi \cdot h_0 f,
  \end{equation*}
  which directly yields
  \begin{equation*}
    \boldsymbol{\iota^*}([fg])
    =
    \boldsymbol{\iota^*}([f\prol \boldsymbol{\iota^*}g])
    =
    \boldsymbol{\iota^*}f \cdot \boldsymbol{\iota^*}g.
  \end{equation*}
  The compatibility of $\prol$ in the setting $M=M_\nice$ in the
  notation of \cite{gutt.waldmann:2010a} is clear since it is just a
  pull-back.  In addition, we get a unique induced formal Poisson
  structure on $M_\red$ via
  \begin{equation*}
    \pi^*\{u,v\}_\red
    =
    \boldsymbol{\iota^*}\{[\prol\pi^*u],[\prol\pi^*v]\}_\hbar.
  \end{equation*}
  Antisymmetry is clear and also the Jacobi identity follows directly,
  where we omit the sign for the equivalence classes:
  \begin{align*}
    \pi^* \{u,\{v,w\}_\red\}_\red
    & =
    \boldsymbol{\iota^*}\{\prol\pi^* u , \prol\boldsymbol{\iota^*}
    \{\prol\pi^* v,\prol \pi^* w\}_\hbar \}_\hbar \\
    & =
    \boldsymbol{\iota^*}\left(\{\{\prol\pi^* u , 
    \prol\pi^* v\}_\hbar,\prol \pi^* w \}_\hbar 
    + \{\prol\pi^* v , 
    \{\prol\pi^* u,\prol \pi^* w\}_\hbar \}_\hbar\right)   \\
    & =
    \pi^*\left( \{\{u,v\}_\red,w\}\}_\red
    +\{v,\{u,w\}_\red\}_\red \right).
  \end{align*}
  Concerning the Leibniz identity we get
  \begin{align*}
    \pi^*\{u,vw\}_\red
    & =
    \boldsymbol{\iota^*}\{\prol\pi^* u, \prol(\pi^*v)\prol(\pi^*w)\}_\hbar \\
    & =
    \boldsymbol{\iota^*} \left( 
    \{\prol\pi^* u, 	\prol(\pi^*v)\}_\hbar\; \prol(\pi^*w)
    +
    \prol(\pi^*v) \;\{\prol\pi^* u, \prol(\pi^*w)\}_\hbar
    \right)  \\
    & =
    \pi^*(v\{u,w\}_\red + \{u,v\}_\red w)
  \end{align*}
  since $\boldsymbol{\iota^*}(f \prol\phi) = \boldsymbol{\iota^*}(f)\phi$.
\end{proof}

Now we want to show that the reduction procedure is compatible with
equivalences, i.e. that equivalent formal Poisson structures with
formal momentum maps are reduced to equivalent reduced Poisson
structures.

\begin{proposition}
  \label{prop:HomPerRedofEquivalences}
  Let $T = \exp(X_\hbar) \colon (\pi_\hbar,J_\hbar) \rightarrow
  (\pi_\hbar',J_\hbar')$ be an equivalence of formal invariant Poisson
  structures with momentum maps, i.e. $X_\hbar \in \hbar
  \Secinfty(TM)[[\hbar]]$ such that
  \begin{equation}
    \label{eq:equivformalpoissonmomentum}
    T \pi_\hbar
    =
    \pi_\hbar' 
    \quad \text{and} \quad 
    T \circ J_\hbar
    =
    J_\hbar'.
  \end{equation} 
  Then one has even $X_\hbar \in \hbar
  \Secinfty(TM)^\group{G}[[\hbar]]$ and
  \begin{equation}
	  \label{eq:HomPerRedofEquivalences}
    T_\red
    =
    (\pi^*)^{-1} \circ \boldsymbol{\iota^*}' \circ T \circ \prol \circ \pi^*
  \end{equation}
  is an equivalence between the reduced formal Poisson structures 
  $\pi_\red$ and $\pi_\red'$.
\end{proposition}
\begin{proof}
  The proof is analogue to the case of star products in
  \cite[Lemma~4.3.1]{reichert:2017b}. At first, as in
  \cite[Prop.~6.2.20]{waldmann:2007a} one can show that $T \pi_\hbar =
  \pi_\hbar'$ is equivalent to
  \begin{equation*}
    T\{f,g\}_\hbar
    =
    \{Tf,Tg\}_\hbar'.
  \end{equation*}
  But then \eqref{eq:equivformalpoissonmomentum} implies
  \begin{equation*}
    \Lie_{\xi_M} Tf
    =
    \{Tf, J_\hbar'(\xi)\}_\hbar'
    =
    T\{f,J_\hbar(\xi)\}_\hbar
    =
    T \Lie_{\xi_M} f.
  \end{equation*} 
  In particular, this yields $[\xi_M, X_\hbar] = 0$ and thus the
  invariance of $X_\hbar$. In addition, recall from
  Proposition~\ref{prop:reducedpoissonvianormalizer} that we have an
  isomorphism of Poisson algebras
  \begin{equation*}
    \left(\Cinfty(M_\red)[[\hbar]],\pi_\red\right) 
    \cong
    \frac{\mathcal{B}_\hbar}{\mathcal{J}_\hbar}.
  \end{equation*}
  By \cite[Prop.~6.2.7]{waldmann:2007a} we know that $T$ is an
  automorphism with respect to the pointwise product, thus we see
  directly from the definition of the deformed Koszul differential
  that
  \begin{equation*}
    T \circ \del_\hbar
    =
    \del_\hbar' \circ T
    \quad \quad 
    \Longrightarrow
    \quad \quad
    T \colon 
    \mathcal{J}_\hbar
    \stackrel{\cong}{\longrightarrow}
    \mathcal{J}_\hbar'.
  \end{equation*}
  Analogously, we have for $j'\in \mathcal{J}_\hbar'$ with $j =
  T^{-1}j' \in \mathcal{J}_\hbar$ and $f\in \mathcal{B}_\hbar$
  \begin{equation*}
    \{Tf, j'\}_\hbar'
    =
    T \{f,j\}_\hbar 
    \in 
    T \mathcal{J}_\hbar
    =
    \mathcal{J_\hbar}'
    \quad \quad
    \Longrightarrow
    \quad \quad
    T \colon
    \mathcal{B}_\hbar
    \stackrel{\cong}{\longrightarrow}
    \mathcal{B}_\hbar'.
  \end{equation*}
  Thus $T_\red$ establishes an isomorphism of the spaces
  $\mathcal{B}_\hbar/\mathcal{J}_\hbar$ and
  $\mathcal{B}_\hbar'/\mathcal{J}_\hbar'$.  It remains to check the
  compatibility with the Poisson bracket:
  \begin{align*}
    \pi^*T_\red\{u,v\}_\red
    & =
    \boldsymbol{\iota^*}' T \prol \boldsymbol{\iota^*} 
    \{\prol\pi^* u,\prol\pi^*v\}_\hbar  \\
    & =
    \boldsymbol{\iota^*}' T  
    \{\prol\pi^* u,\prol\pi^*v\}_\hbar 
  \end{align*}
  since $T$ maps the kernel of $\boldsymbol{\iota^*}$ into the kernel
  of $\boldsymbol{\iota^*}'$. On the other hand, we get
  \begin{align*}
    \pi^*\{T_\red u, T_\red v\}_\red'
    & = 
    \boldsymbol{\iota^*}' \{ \prol \boldsymbol{\iota^*}' T \prol \pi^* u,
    \prol \boldsymbol{\iota^*}' T \prol \pi^* v\}_\hbar ' \\
    & =
    \boldsymbol{\iota^*}' \{  T \prol \pi^* u, T \prol \pi^* v\}_\hbar ' 
  \end{align*}
  since we take on the right hand side the bracket in
  $\mathcal{B}_\hbar' / \mathcal{J}_\hbar'$ where $[\prol
    \boldsymbol{\iota^*}' f] = [f]$.  Thus the compatibility with the
  brackets is shown.  It remains to show that $T_\red$ is of the form
  $T_\red = \exp(X_{\red,\hbar})$ for some vector field
  $X_{\red,\hbar}\in\hbar\Secinfty(TM_\red)[[\hbar]]$.  Since $T =
  \exp(X_\hbar)$ we know that $T_\red$ is a formal power series of
  $\mathbb{C}[[\hbar]]$-linear operators starting with $\id +
  \hbar(\dots)$. We can write $T_\red = \exp(\hbar D)$ via
  \begin{equation*}
    \hbar D 
    =
    \sum_{s=0}^\infty \frac{(-1)^{s+1}}{s} \left(T - \id\right)^s.
  \end{equation*}
  Again by \cite[Prop.~6.2.7]{waldmann:2007a} it suffices to show
  $T_\red(uv) = T_\red(u)T_\red(v)$, which directly implies $T_\red =
  \exp(X_\red)$ for some vector field $X_\red \in
  \hbar\Secinfty(TM_\red)[[\hbar]]$.  But this is clear since each of
  the involved maps in the definition of $T_\red$ is compatible with
  the pointwise product: The maps $\prol, \pi^*$ and $(\pi^*)^{-1}$
  since they resp. their inverses are pull-backs, the map $T$ since
  $T= \exp(X_\hbar)$ and $\boldsymbol{\iota^*}$ by
  Proposition~\ref{prop:reducedpoissonvianormalizer}.
\end{proof}

%
% Bibliography
%
\bibliographystyle{nchairx}

\end{document}